\documentclass[a4paper,reqno,11pt]{amsart}
\usepackage{amsmath,amssymb,amscd,mathrsfs,amscd,wasysym,color,xcolor}
\usepackage{verbatim}
\usepackage[linktocpage=true]{hyperref}
\usepackage{tikz-cd}

\newdimen\hoogte    \hoogte=12pt    
\newdimen\breedte   \breedte=14pt   
\newdimen\dikte     \dikte=0.5pt    

\newenvironment{Young}{\begingroup
       \def\vr{\vrule height0.8\hoogte width\dikte depth 0.2\hoogte}
       \def\fbox##1{\vbox{\offinterlineskip
                    \hrule height\dikte
                    \hbox to \breedte{\vr\hfill##1\hfill\vr}
                    \hrule height\dikte}}
       \vbox\bgroup \offinterlineskip \tabskip=-\dikte \lineskip=-\dikte
            \halign\bgroup &\fbox{##\unskip}\unskip  \crcr }
       {\egroup\egroup\endgroup}
\def\diagram#1{\relax\ifmmode\vcenter{\,\begin{Young}#1\end{Young}\,}\else%
              $\vcenter{\,\begin{Young}#1\end{Young}\,}$\fi}

\theoremstyle{plain}
\newtheorem{thm}{Theorem}[section]
\newtheorem{lem}[thm]{Lemma}

\newtheorem{prp}[thm]{Proposition}
\newtheorem{dfn}[thm]{Definition}
\newtheorem{cor}[thm]{Corollary}

\newtheorem{exa}[thm]{Example}

	\newtheorem{alphatheorem}{Theorem}

\theoremstyle{remark}
\newtheorem{rmk}[thm]{Remark}

\numberwithin{equation}{section}

\def\genE{{\mathsf{E}}}
\def\genF{{\mathsf{F}}}
\def\genK{{\mathsf{K}}}

\newcommand{\al}{\alpha}

\newcommand{\ep}{\varepsilon}

\newcommand{\Z}{{\mathbb{Z}}}

\newcommand{\C}{{\mathbb{C}}}

\def\ol#1{\overline{#1}}

\newcommand{\Ui}{{\mathbf U}^\imath}

\newcommand{\Id}{\operatorname{Id}\nolimits}

\newcommand{\zero}{{\bar{0}}}
\newcommand{\one}{{\bar{1}}}

\newcommand{\iso}{{\mathsf{iso}}}
\newcommand{\niso}{{\mathsf{n}\text{-}\mathsf{iso}}}

\newcommand{\yy}{\mbox{\RIGHTcircle}}

\newcommand{\ad}{{\mathrm{ad}}}

\newcommand{\Ad}{{\mathrm{Ad}}}

\newcommand{\QED}{\rule{0.4em}{2ex}}

\renewcommand{\bar}[1]{\overline{#1}}

\everymath=\expandafter{\the\everymath\displaystyle} 

\allowdisplaybreaks[1]

\input xy
\xyoption{all}

\begin{document}

\title[ Quantum supersymmetric pairs and the Serre relations via $\mathrm i$Hopf algebras]{ Quantum supersymmetric pairs and the Serre relations via $\mathrm i$Hopf algebras }



\begin{abstract} 

We study iHopf algebras 
associated with quantum supergroups of basic type. For a quantum supersymmetric pair $(\mathbf{U}, \mathbf{U}^\imath)$, we realize $\mathbf{U}^\imath$ and $\mathbf{U}$  as the iHopf algebras of the Borel quantum group ${\hat{\mathbf{U}}}_q^{\geq0}$ and the tensor product algebra ${\hat{\mathbf{U}}}_q^{\geq0}\otimes {\hat{\mathbf{U}}}_q^{\geq0}$, respectively, while the coideal subalgebra structure is encoded by an embedding between the iHopf algebras. Moreover, we derive an explicit conversion formula between the iHopf multiplication and the original multiplication, 
{which provides a general mechanism transforming the defining relations of quantum (super)groups into relations of iquantum (super)groups.}
In particular, all the Serre relations that occur in quasi-split iquantum supergroups of basic type are characterized.
\end{abstract}


\author[Jiayi Chen]{Jiayi Chen}
\address{Department of Mathematics, Shantou University, Shantou 515063, P.R.China}
\email{chenjiayi@stu.edu.cn}

\author[Shiquan Ruan]{Shiquan Ruan}
\address{ School of Mathematical Sciences,
Xiamen University, Xiamen 361005, P.R.China}
\email{sqruan@xmu.edu.cn}

\author[Hongying Zhu]{Hongying Zhu}
\address{ School of Mathematical Sciences,
Xiamen University, Xiamen 361005, P.R.China}
\email{hongyingz@stu.xmu.edu.cn}

\dedicatory{Dedicated to Professor Yanan Lin on the occasion of his seventieth birthday}

\maketitle

\setcounter{tocdepth}{1}  \tableofcontents

\section{Introduction }
\subsection{Background}{Quantum groups} originated in the field of physics, with Faddeev playing a pivotal role in their development for constructing and solving `integrable' quantum systems. 
The theory of quantum groups, introduced by Drinfeld \cite{Dri85} and Jimbo \cite{Jim85} in the 1980s, provides a deformation of universal enveloping algebras of semisimple Lie algebras.
Their super counterparts, known as quantum superalgebras or quantum supergroups, arise naturally from Lie superalgebras. Lie superalgebras, classified by Kac \cite{Kac77} into classical types and Cartan types, describe supersymmetries in physics and have found applications in areas ranging from string theory to integrable systems; see \cite{JGZ90}. The quantization of universal  enveloping superalgebras leads to quantum supergroups, whose representation theory and structural properties have been extensively studied by Yamane \cite{Ya94,Ya99} and others \cite{CHW16,LWY22}.

{Quantum symmetric pairs} $(\mathbf{U},\mathbf{U}^{\imath})$   are quantum deformations of symmetric pairs associated with Satake diagrams.  Here $\mathbf{U}$  denotes the quantum group, and $\mathbf{U}^\imath$  is a coideal subalgebra of $\mathbf{U}$, commonly called an iquantum group. The theory of quantum symmetric pairs was systematically developed by Letzter in the finite type setting; see \cite{Le99,Le02} for historical remarks and references. The theory was subsequently extended to the Kac--Moody setting by Kolb \cite{Ko14}, thereby unifying a number of special cases previously considered in the literature.
{Over the past decade, several fundamental constructions for quantum groups
have been extended to iquantum groups. In particular,  Bao and Wang
developed a theory of icanonical bases, (quasi) $K$-matrices for quantum symmetric pairs and
established a new approach to Kazhdan--Lusztig theory of type \(B\); see
\cite{BW18a,BW18b}. The theory of (quasi) \(K\)-matrices was further
developed by Balagović and Kolb, who constructed a universal $K$-matrix for general quantum symmetric pair coideal subalgebra of $\mathbf{U}_q(\mathfrak{g})$ for $\mathfrak{g}$ of finite type; see \cite{BK19}. }


The theory of quantum symmetric pairs has since been extended to the super setting, giving rise to the notion of {quantum supersymmetric pairs}. In this direction, Kolb and Yakimov extended the construction of quantum symmetric pairs to Nichols algebras of diagonal type, which can be applied uniformly to quantum supergroups; see \cite{KY20}. Chung \cite{Ch19} studied symmetric pairs for quantum covering algebras, while Chen and Luo \cite{CL23} extended the theory of iSchur algebras to the super setting. 
More recently, Shen and Wang \cite{Sh25, SW25} developed a theory of quantum supersymmetric pairs associated with super Satake diagrams for arbitrary Dynkin diagrams of finite-dimensional basic Lie superalgebras. Following the terminology  from the theory of real Lie groups, a quantum (super)symmetric pair and its associated iquantum (super)group are called quasi-split if the corresponding Satake diagram has no black nodes.  They are called split if, in addition, the involution on the diagram is trivial.

Serre relations play a crucial role in the theory of quantum groups, as they yield concise presentations of the corresponding algebras. Naturally, one also investigates the Serre relations for iquantum groups and quantum supergroups.
 In finite type, Letzter established Serre-type presentations for general iquantum groups; see \cite{Le02,Le03}. In the Kac--Moody setting, however, early treatments, including \cite{Ko14}, left some of the Serre relations implicit. For  $|a_{i,j}| \leq 3$, subsequent work produced explicit relations among the generators  $B_i$ and $ B_j$; see  \cite{Le03,Ko14,BK15,BK19}. These formulas, however, become increasingly complicated as  $|a_{i,j}|$ grows. The general case was finally settled in \cite{CLW21}, which provided a uniform and explicit Serre presentation for quasi-split iquantum groups of arbitrary Kac--Moody type with general parameters. In the split case, a concise formulation requires the use of idivided powers in place of ordinary divided powers, and their treatment is substantially more involved.

Owing to their $\Z_2$-graded structure, quantum supergroups admit a richer family of Serre relations than ordinary quantum groups: the precise form of the relations depends on the parities of the simple roots. 
Yamane \cite{Ya94} established {Serre presentations for quantum supergroups} of basic type associated with arbitrary Dynkin diagrams. He subsequently described, in \cite{Ya99}, defining relations for the Chevalley generators of affine quantized universal enveloping superalgebras. 
A distinctive feature of the super setting is the existence of many different types of Serre relations involving 2, 3 or 4 Chevalley generators; see \cite{Ya94,Ya99,CHW16}. 
Extending this precision to the iquantum setting, a fundamental objective is to obtain a concrete presentation for $\mathbf{U}^\imath$. Such presentations are essential for constructing the relevant bar involution and for determining the parameter constraints on the generators; see \cite{BK15}.

The iHopf algebra, introduced in \cite{CLPRW25}, is an associative algebra associated with a Hopf algebra equipped with a Hopf pairing. In the setting of quantum groups, the iHopf algebra attached to a Borel subalgebra realizes quasi-split universal iquantum groups of Kac--Moody type; moreover, the iHopf algebra of diagonal type realizes the corresponding quantum groups and encodes the coideal subalgebra structures of quantum symmetric pairs. Since the iHopf algebra has the same underlying vector space as the original Hopf algebra but has an adjusted multiplication structure, various structures on the Hopf algebra can often be transferred to the iHopf setting. For example, Serre relations, braid group actions, and canonical bases on quantum groups admit corresponding analogues on iquantum groups; see \cite{CLPRW25, CLPRW25}. Quantum supergroups are naturally Hopf superalgebras. After adjoining suitable parity elements, they may be regarded as ordinary Hopf algebras. This allows us to apply the iHopf-algebra construction to quantum supergroups and thereby study iquantum supergroups and quantum supersymmetric pairs.

\subsection{Goal} In this paper, we study the iHopf algebras associated with quantum supergroups of basic type and provide the conversion formulas between the original multiplication and the iHopf multiplication.
We prove that the quantum supergroup is isomorphic to the iHopf algebra of diagonal type, whereas the quasi-split iquantum supergroup is isomorphic to the iHopf algebra of the Borel subalgebra. Moreover, an embedding between these two iHopf algebras encodes the coideal subalgebra structure, thereby yielding a realization of quantum supersymmetric pairs.

Consequently, by converting the multiplication in the Serre relations of quantum supergroups into the iHopf multiplication via the established formulas, we derive the Serre relations for quasi-split iquantum supergroups of basic type. In particular, when all generators are even, this covers the Serre relations for iquantum groups.


\subsection{Main results} Let {$\C$} be the field of complex numbers, and  $\C(q)$ be the field of  rational functions in one variable $q$ over  {$\C$}. In the super setting, $\mathbb{I}=\mathbb{I}_\zero\cup\mathbb{I}_\one$ consists of even and odd simple roots. 
Let ${\hat{\mathbf{U}}}_q^{\geq0}$ be the Borel quantum supergroup generated by $\genE_i,\genK_i,\varrho$ ($i\in \mathbb{I}$), endowed with a non-degenerate Hopf pairing $\langle-,-\rangle$. Denote by $p(\cdot)$ the parity function.
Following \cite{CLPRW25}, for a quasi-split super Satake diagram $(\mathbb{I},\tau)$, we define the iHopf algebra $({\hat{\mathbf{U}}}_q^{\geq0})_\tau^{ \imath}$, which shares the same vector space as ${\hat{\mathbf{U}}}_q^{\geq0}$, equipped with the new multiplication 
\begin{align*}
   a\circ b=\sum  \langle\tau(b_{(2)}),a_{(1)}\rangle a_{(2)}b_{(1)}.
\end{align*}
Under the new multiplication, we establish basic relations for the iHopf algebra  $({\hat{\mathbf{U}}}_q^{\geq0})_\tau^{ \imath}$; see Proposition \ref{KE}. Furthermore, to derive the Serre relations for $({\hat{\mathbf{U}}}_q^{\geq0})_\tau^{ \imath}$, we provide an explicit conversion formula for products of generators $\genE_i$ between the quantum supergroup ${\hat{\mathbf{U}}}_q^{\geq0}$ and the iHopf algebra $({\hat{\mathbf{U}}}_q^{\geq0})_\tau^{ \imath}$.

For $I=(i_1,\cdots,i_l)\in {\mathbb{I}}^l$, denote 
 $$E_I:=\genE_{i_1}\genE_{i_2}\cdots \genE_{i_l}=\prod\limits_{u=1}^l \genE_{i_u}, \quad E_I^{\circ}:=\genE_{i_1}\circ\genE_{i_2}\circ\cdots \circ\genE_{i_l}=\mathop{\circ\prod}\limits_{u=1}^l \genE_{i_u}.$$ Denote by $\mathcal{M}_{s}(S)$ the collection of all sets $M$ consisting of $s$  ordered pairs in $S=\{1,2,\cdots,l\}$ with suitable conditions as in \eqref{def:m}. 
 For $M\in \mathcal{M}_{s}(S)$, set $$T(M)=\{x,y\in S\mid (x,y)\in M\}.$$
 For each $(x,y)\in M$, we define functions $f(x,y)$, $g(x,y)$ and $\varphi(x,y)$ in \eqref{def1}-\eqref{def3}. Then we obtain the following conversion formula between the original multiplication and iHopf multiplication.

 \begin{alphatheorem}[Theorem \ref{formula}]\label{A}
     Suppose that $I=(i_1,\cdots,i_l)\in {\mathbb{I}}^l$. Then
     $$E_{I}=E^{\circ}_{I}+o(E^{\circ}_{I}),$$
     where $o(E^{\circ}_{I})$ is given by
     {{\begin{align*}
    \sum\limits_{s=1}^{\lfloor \frac{l}{2} \rfloor}\sum\limits_{M\in \mathcal{M}_s(S)} (-1)^s\Big(\mathop{\circ\prod}\limits_{(x, y)\in M}\varphi(x,y)\cdot \varrho^{p(i_{x})}\circ \genK_{i_{y}}\Big)\circ(\mathop{\circ\prod}\limits_{\substack{1\leq u\leq l\\u\notin T(M)}}\genE_{i_u}).
     \end{align*}}}
 \end{alphatheorem}

Given the  Hopf algebra structure on the quantum supergroup ${{\mathbf{U}}}$ in \eqref{defU}, the iHopf algebra $({\hat{\mathbf{U}}}_q^{\geq0}\otimes{\hat{\mathbf{U}}}_q^{\geq0})^{\imath}$ associated with the tensor product Hopf algebra ${\hat{\mathbf{U}}}_q^{\geq0}\otimes{\hat{\mathbf{U}}}_q^{\geq0}$ and a twisted Hopf pairing is isomorphic to ${{\mathbf{U}}}$.  

Under the assumption that $\tau i\neq i$ whenever $i\in \mathbb{I}_\iso$, we provide a realization of quantum supersymmetric pairs $({\mathbf{U}},{\mathbf{U}}^\imath)$ and the coideal subalgebra structure via iHopf algebras.

 \begin{alphatheorem}[Theorem \ref{Phi}, Proposition \ref{xi}, Theorem \ref{Phii}]
    There are algebra {isomorphisms} $\Phi: {{\mathbf{U}}}\to ({\hat{\mathbf{U}}}_q^{\geq0}\otimes{\hat{\mathbf{U}}}_q^{\geq0})^{\imath}$ and $\Phi^{\imath}_\tau:({\hat{\mathbf{U}}}_q^{\geq0})_\tau^{ \imath}\to\mathbf{U}^{\imath}$,
    and an algebra embedding $\xi^{\imath}_\tau: ({\hat{\mathbf{U}}}_q^{\geq0})^{ \imath}_\tau\to ({\hat{\mathbf{U}}}_q^{\geq0}\otimes{\hat{\mathbf{U}}}_q^{\geq0})^{\imath}$
      such that the following diagram commutes
\begin{equation*}
\begin{tikzcd}
({\hat{\mathbf{U}}}_q^{\geq0})_\tau^{ \imath} \arrow[r, "{\xi^{ \imath}_\tau}"] \arrow[d, "{\Phi}_\tau^{\imath}"] 
& ({\hat{\mathbf{U}}}_q^{\geq0} \otimes {\hat{\mathbf{U}}}_q^{\geq0})^{ \imath} \arrow[d, "{\Phi}^{-1}"] \\
\Ui \arrow[r, hookrightarrow] 
& {{\mathbf{U.}}}
\end{tikzcd} 
\end{equation*} 
 \end{alphatheorem}

The iHopf realization of iquantum supergroups admits various applications. As an application of Theorem \ref{A}, we derive the Serre relations for quasi-split iquantum supergroups. To simplify the conversion formula, we employ formal symbols $\mathfrak{e}_{i_{t}} (1\leq t\leq l)$. (Do not treat $\mathfrak{e}_{i_{t}}$ as algebra generators!) Denoting $\genE_i^{\star}=\genE_i+\mathfrak{e}_{i}$ and normalizing the multiplication via the matching rule \eqref{evaluate}, we reformulate the Serre relations for the iHopf algebra  $({\hat{\mathbf{U}}}_q^{\geq0})_\tau^{ \imath}$ of basic type as follows.

\begin{alphatheorem}[Theorem \ref{star}]
For any $\tau$ with $\tau i\neq i$ if $i\in \mathbb{I}_\iso$, the superalgebra $({\hat{\mathbf{U}}}_q^{\geq0})_\tau^{ \imath}$  of basic type satisfies the following relations whenever the given Dynkin subdiagram appears: 

\begin{enumerate}
\item For $i,j\in \mathbb{I}_\one$ with $a_{ij}=0$, 
$$\genE_i^{\star}\circ\genE_j^{\star}=-\genE_j^{\star}\circ\genE_i^{\star}.$$
\item For $i\in \mathbb{I}_\zero\cup \mathbb{I}_\niso$ and $i\neq j$ with $a_{i,j}\in\{0,-1,-2\}$, 
$$\Ad_q \genE_i^{{\star} ^{1-a_{i,j}}}(\genE_j^{\star})=0.
$$
\item For  $sgn_{ij}\neq sgn_{jk}$,
$$\hspace{.75in}\xy
(-10,0)*{\odot};(0,0)*{\otimes}**\dir{-};(0,0)*{\otimes};(10,0)*{\odot}**\dir{-};
(-10,-4)*{\scriptstyle i};(0,-4)*{\scriptstyle j};(10,-4)*{\scriptstyle k};
\endxy$$ 
$$
\Ad_q \genE_j^{\star}\circ\Ad_{q}\genE_k^{\star}\circ \Ad_{q} \genE_j^{\star} (\genE_i^{\star})=0.
$$
\item For $\tau j=k$,
$$\xy
(0,0)*{\textup\fullmoon};(7,5)*{\otimes}**\dir{-};(0,0)*{\textup\fullmoon};(7,-5)*{\otimes}**\dir{-};(7,5)*{\otimes};(7,-5)*{\otimes}**\dir{=};
(-4,0)*{\scriptstyle i};(10,5)*{\scriptstyle j};(10,-5)*{\scriptstyle k};
\endxy\quad $$
$$\Ad_q \genE_k^{\star}\circ \Ad_q \genE_j^{\star} (\genE_i^{\star})=\Ad_q \genE_j^{\star} \circ \Ad_q \genE_k^{\star}(\genE_i^{\star}).$$
\end{enumerate}
\end{alphatheorem}

\subsection{Organization}
{The paper is organized as follows.} In Section \ref{Preliminaries}, we review the root data of basic Lie superalgebras and recall the definitions of quantum supergroups of basic type, quantum supersymmetric pairs and iHopf algebras. In Section \ref{Hopfpair}, we establish a Hopf pairing  of the Borel quantum supergroup ${\hat{\mathbf{U}}}_q^{\geq0}$. Section \ref{iHopf} is devoted to the study of the iHopf algebra $({\hat{\mathbf{U}}}_q^{\geq0})_{\tau}^\imath$ and its fundamental  relations,  where we also formulate a conversion formula between the original multiplication and iHopf multiplication. In Section \ref{realization}, we provide a realization of the quantum supersymmetric pairs $({\mathbf{U}},{\mathbf{U}}^\imath)$ of quasi-split type via iHopf algebras $({\hat{\mathbf{U}}}_q^{\geq0}\otimes{\hat{\mathbf{U}}}_q^{\geq0})^{\imath}$ and $({\hat{\mathbf{U}}}_q^{\geq0})_\tau^{\imath}$. In Section \ref{application}, we show that the multiplication conversion formulas derived in Section \ref{iHopf} yield explicit Serre relations for quasi-split iquantum supergroups.

\section{Preliminaries}\label{Preliminaries}

In this section, we review the definitions and fundamental properties of basic Lie superalgebras  $\mathfrak{g}$,  quantum supergroups, quantum supersymmetric pairs and iHopf algebras.

\subsection{Root data}
Set $\mathbb{I}=\mathbb{I}_\zero\cup\mathbb{I}_\one=\{1,2,\cdots,r\}$. Let $\mathfrak{g}=\mathfrak{g}_\zero\oplus\mathfrak{g}_\one$ be a finite-dimensional complex simple Lie superalgebra of basic type such that $\mathfrak{g}_{\ol{1}}\neq 0$, and let $$\Pi=\Pi_\zero\cup\Pi_\one=\{\alpha_1, \alpha_2,\cdots \alpha_r\}$$ 
with $r$ the rank of $\mathfrak{g}$, be the simple roots of $\mathfrak{g}$. 

Denote by $\Phi_\zero$ and $\Phi_\one$ the sets of the even and odd roots, respectively, and $\Phi^+=\Phi_\zero^+\cup \Phi_\one^+$ (resp. $\Phi^-=\Phi_\zero^-\cup \Phi_\one^-$ ) the associated positive (resp. negative) root system. Decompose $\Phi_\one=\Phi_{\iso}\cup \Phi_{\niso}$, where $\Phi_{\iso}$ (resp. $\Phi_\niso$) is the set of isotropic (resp. non-isotropic) odd roots. Decompose $\Pi_\one=\Pi_\iso\cup \Pi_\niso$ and $\mathbb{I}_\one=\mathbb{I}_\iso\cup \mathbb{I}_\niso$ accordingly.
For $\al\in\Phi=\Phi_\zero\cup\Phi_\one$, we write 
\begin{align*}
    p(\al)=s,~\mbox{ if }~\al\in\Phi_s.
\end{align*}
 We often write 
 \begin{equation}
p(i):=p(\alpha_i),
 \end{equation}
 for $i\in\mathbb{I}$. 
 By linearity, we also have a parity function $p(\cdot)$ on the root lattice $X=\Z\Pi$.
 Set $\mathfrak{n^{\pm}}=\oplus_{\al\in\Phi^{\pm}}\mathfrak{g}_{\al}$. Then the Lie superalgebra $\mathfrak{g}$ has a triangular decomposition $$\mathfrak{g=n^+\oplus h \oplus n^-}.$$

The basic Lie superalgebras carry a non-degenerate even supersymmetric bilinear form associated with their Cartan matrices \cite{Kac77}. Let $(A,\mu)$ denote the corresponding Cartan data, where $A=(a_{i,j})$  is the Cartan matrix and $\mu\subset\mathbb{I}$ specifies the parity of the generators. The matrix $A$ is assumed to be symmetrizable; that is, there exist non-zero rational numbers  $d_1, d_2,\cdots,d_r$ such that $$d_i a_{i,j}=d_j a_{j,i}.$$ Since the non-degenerate supersymmetric invariant bilinear form $(,)$ on $\mathfrak{g}$  is unique up to a scalar factor, and its restriction to the Cartan subalgebra remains non-degenerate, the symmetrizing factors $d_i$ are determined up to an overall constant. We may therefore assume without loss of generality that $d_1 =1$.

 If $(\alpha_i,\alpha_j)\neq 0$, define
\begin{align*}
    sgn_{ij}=\text{sign of} ~(\alpha_i,\alpha_j),
\end{align*}
which provides the additional information needed to recover a Cartan matrix from its Dynkin diagram.

 Recall that the set of simple coroots 
$\Pi^{\vee}=\{h_i~|~i\in\mathbb{I}\}\subset \mathfrak{h}$ 
is given by 
\begin{align*}
    \al_j(h_i)=a_{i,j}
\end{align*}
for $i,j\in\mathbb{I}$. Let $Y=\Z\Pi^{\vee}$ denote the coroot lattice.

Let $\mathfrak{g}(A,\{s\})$ be the Lie superalgebra whose root system is defined with respect to a special Borel sub-superalgebra containing at most one odd simple root; such a system is called a $distinguished$ root system.

 As in \cite[Section 2]{CHW16}, Table~1 below lists Dynkin diagrams for arbitrary choices of $\Phi^+$ (for type $A$-$D$), with nodes labelled according to the simple roots. Diagrams marked $(\star)$ in types $F(3|1)$ and $G(3)$ are ${distinguished}$
($F(3|1)$ is often referred to as $F(4)$ in the literature). Nodes correspond to even, odd isotropic, or odd non-isotropic simple roots, denoted  $\fullmoon$, $\otimes$, and $\newmoon$, respectively; $\odot$ indicates a root that is either odd isotropic or even, while  $\yy$ indicates one that is either odd non-isotropic or even.

 \begin{table}
  \vspace{.5cm}
 {Table 1: Dynkin diagrams for general simple systems}
  \vspace{.5cm}
\begin{tabular}{|c|c|}\hline
$A(m,n)$&
$$ \xy
(-30,0)*{\odot};(-20,0)*{\odot}**\dir{-};(-15,0)*{\cdots};(-10,0)*{\odot};(0,0)*{\odot}**\dir{-};
(0,0)*{\odot};(10,0)*{\odot}**\dir{-};
(15,0)*{\cdots};(20,0)*{\odot};(30,0)*{\odot}**\dir{-};
(-30,-4)*{\scriptstyle 1};(-20,-4)*{\scriptstyle 2};(-10,-4)*{\scriptstyle n};(0,-4)*{\scriptstyle n+1};(10,-4)*{\scriptstyle n+2};
(20,-4)*{\scriptstyle m+n};(30,-4)*{\scriptstyle m+n+1};
(0,-8)*{};(0,8)*{};
\endxy$$
\\\hline
$B(m,n+1)$&
$$ \xy
(-30,0)*{\odot};(-20,0)*{\odot}**\dir{-};(-15,0)*{\cdots};(-10,0)*{\odot};(0,0)*{\odot}**\dir{-};
(0,0)*{\odot};(10,0)*{\odot}**\dir{-};
(15,0)*{\cdots};(20,0)*{\odot};(30,0)*{\yy}**\dir{=};(25,0)*{>};
(-30,-4)*{\scriptstyle 1};(-20,-4)*{\scriptstyle 2};(-10,-4)*{\scriptstyle n};(0,-4)*{\scriptstyle n+1};(10,-4)*{\scriptstyle n+2};
(20,-4)*{\scriptstyle m+n};(31,-4)*{\scriptstyle m+n+1};
(0,-8)*{};(0,8)*{};
\endxy$$
\\\hline
$C(n+1)$&
$$ \xy
(-20,0)*{\odot};(-10,0)*{\odot}**\dir{-};(-5,0)*{\cdots};(0,0)*{\odot};(10,0)*{\odot}**\dir{-};
(10,0)*{\odot};(20,0)*{\fullmoon}**\dir{=};(15,0)*{<};
(-20,-4)*{\scriptstyle 1};(-10,-4)*{\scriptstyle 2};
(10,-4)*{\scriptstyle n};(20,-4)*{\scriptstyle n+1};
(0,-8)*{};(0,8)*{};
\endxy$$
\\
$D(m,n+1)$&
$$\xy
(-30,0)*{\odot};(-20,0)*{\odot}**\dir{-};(-15,0)*{\cdots};
(-10,0)*{\odot};(0,0)*{\odot}**\dir{-};(0,0)*{\odot};(10,0)*{\odot}**\dir{-};
(15,0)*{\cdots};(20,0)*{\odot};(27,5)*{\fullmoon}**\dir{-};(20,0)*{\odot};(27,-5)*{\fullmoon}**\dir{-};
(-30,-4)*{\scriptstyle 1};(-20,-4)*{\scriptstyle 2};(-10,-4)*{\scriptstyle n};(0,-4)*{\scriptstyle n+1};(10,-4)*{\scriptstyle n+2};
(33,5)*{\scriptstyle m+n};(34,-5)*{\scriptstyle m+n+1};
(0,-8)*{};(0,8)*{};
\endxy$$
\\
&
$$\xy
(-30,0)*{\odot};(-20,0)*{\odot}**\dir{-};(-15,0)*{\cdots};
(-10,0)*{\odot};(0,0)*{\odot}**\dir{-};(0,0)*{\odot};(10,0)*{\odot}**\dir{-};
(15,0)*{\cdots};(20,0)*{\odot};(27,5)*{\otimes}**\dir{-};(20,0)*{\odot};(27,-5)*{\otimes}**\dir{-};
(27,5)*{\otimes};(27,-5)*{\otimes}**\dir{=};
(-30,-4)*{\scriptstyle 1};(-20,-4)*{\scriptstyle 2};(-10,-4)*{\scriptstyle n};(0,-4)*{\scriptstyle n+1};(10,-4)*{\scriptstyle n+2};
(33,5)*{\scriptstyle m+n};(34,-5)*{\scriptstyle m+n+1};
(0,-8)*{};(0,8)*{};
\endxy$$
\\\hline
$F(3|1)$&
$$\xy (0,8)*{};(-20,0)*{(\star)};
(-15,0)*{\fullmoon};(-5,0)*{\fullmoon}**\dir{-};
(-5,0)*{\fullmoon};(5,0)*{\fullmoon}**\dir{=};(0,0)*{>};(5,0)*{\fullmoon};(15,0)*{\otimes}**\dir{-};
(-15,-4)*{\scriptstyle 1};(-5,-4)*{\scriptstyle 2};(5,-4)*{\scriptstyle 3};(15,-4)*{\scriptstyle 4};
(0,-8)*{};
\endxy$$
\\&
$$\xy (0,8)*{};
{\ar@3{-}(-15,0)*{\fullmoon};(-5,0)*{\otimes}};(-10,0)*{>};
(-5,0)*{\otimes};(5,0)*{\fullmoon}**\dir{=};(0,0)*{<};(5,0)*{\fullmoon};(15,0)*{\fullmoon}**\dir{-};
(-15,-4)*{\scriptstyle 1};(-5,-4)*{\scriptstyle 2};(5,-4)*{\scriptstyle 3};(15,-4)*{\scriptstyle 4};
(0,-8)*{};
\endxy$$ \hspace{.25in}
$$\xy (0,8)*{};
{\ar@3{-}(-15,0)*{\fullmoon};(-5,0)*{\otimes}};(-10,0)*{>};
(-5,0)*{\otimes};(5,0)*{\fullmoon}**\dir{-};(5,0)*{\fullmoon};(15,0)*{\fullmoon}**\dir{=};(10,0)*{<};
(-15,-4)*{\scriptstyle 1};(-5,-4)*{\scriptstyle 2};(5,-4)*{\scriptstyle 3};(15,-4)*{\scriptstyle 4};
(0,-8)*{};
\endxy$$
\\&
$$\xy (0,8)*{};
{\ar@3{-}(-10,0)*{\otimes};(-5,7)*{\otimes}};(-5,7)*{\fullmoon};(0,0)*{\otimes}**\dir{-};
(-10,0)*{\otimes};(0,0)*{\otimes}**\dir{=};(0,0)*{\otimes};(10,0)*{\fullmoon}**\dir{=};(5,0)*{<};
(-10,-4)*{\scriptstyle 1};(-5,9)*{\scriptstyle 2};(0,-4)*{\scriptstyle 3};(10,-4)*{\scriptstyle 4};
(0,-8)*{};
\endxy$$ \hspace{.25in}
$$\xy (0,8)*{};
(-10,0)*{\otimes};(-5,7)*{\fullmoon}**\dir{-};(-5,7)*{\fullmoon};(0,0)*{\otimes}**\dir{-};
(-10,0)*{\otimes};(0,0)*{\otimes}**\dir{=};(0,0)*{\otimes};(10,0)*{\fullmoon}**\dir{=};(5,0)*{<};
(-10,-4)*{\scriptstyle 1};(-5,9)*{\scriptstyle 2};(0,-4)*{\scriptstyle 3};(10,-4)*{\scriptstyle 4};
(0,-8)*{};
\endxy$$
\\\hline
$G(3)$&
$$ \xy(0,5)*{};(-15,0)*{(\star)};
{\ar@3{-}(0,0)*{\fullmoon};(10,0)*{\fullmoon}};(5,0)*{<};(-10,0)*{\otimes};(0,0)*{\fullmoon}**\dir{-};
(-10,-4)*{\scriptstyle 1};(0,-4)*{\scriptstyle 2};(10,-4)*{\scriptstyle 3};
(0,-8)*{};(0,8)*{};
\endxy$$ \hspace{.25in}
$$ \xy(0,5)*{};
{\ar@3{-}(0,0)*{\otimes};(10,0)*{\fullmoon}};(5,0)*{<};(-10,0)*{\otimes};(0,0)*{\otimes}**\dir{-};
(-10,-4)*{\scriptstyle 1};(0,-4)*{\scriptstyle 2};(10,-4)*{\scriptstyle 3};
(0,-8)*{};(0,8)*{};
\endxy$$
\\&
$$ \xy(0,5)*{};
{\ar@3{-}(0,0)*{\otimes};(10,0)*{\fullmoon}};(5,0)*{<};(-10,0)*{\newmoon};(0,0)*{\otimes}**\dir{-};
(-10,-4)*{\scriptstyle 1};(0,-4)*{\scriptstyle 2};(10,-4)*{\scriptstyle 3};
(0,-8)*{};(0,8)*{};
\endxy$$ \hspace{.25in}
$$\xy (0,5)*{};
{\ar@2{-}(-5,0)*{\otimes};(0,7)*{\fullmoon}};(0,7)*{\fullmoon};(5,0)*{\otimes}**\dir{-};
{\ar@3{-}(-5,0)*{\otimes};(5,0)*{\otimes}};
(-5,-4)*{\scriptstyle 1};(0,9)*{\scriptstyle 2};(5,-4)*{\scriptstyle 3};
(0,-8)*{};(0,8)*{};
\endxy$$
\\\hline
$D(2|1;\al)$&
$$\xy
(0,0)*{\otimes};(7,5)*{\fullmoon}**\dir{-};(0,0)*{\otimes};(7,-5)*{\fullmoon}**\dir{-};(2,4)*{\scriptstyle -1};(2,-4)*{\scriptstyle 1+\al};
(-3,0)*{\scriptstyle 1};(10,5)*{\scriptstyle 2};(10,-5)*{\scriptstyle 3};
(0,8)*{};(0,-8)*{};
\endxy$$ \hspace{.25in}
$$\xy
(0,0)*{\otimes};(7,5)*{\fullmoon}**\dir{-};(0,0)*{\otimes};(7,-5)*{\fullmoon}**\dir{-};(2,4)*{\scriptstyle -\al};(2,-4)*{\scriptstyle 1+\al};
(-3,0)*{\scriptstyle 1};(10,5)*{\scriptstyle 2};(10,-5)*{\scriptstyle 3};
(0,8)*{};(0,-8)*{};
\endxy$$
\\$(\al\in\Z_{>0})$&
$$\xy
(-5,0)*{\otimes};(0,7)*{\otimes}**\dir{-};
(-5,0)*{\otimes};(5,0)*{\otimes}**\dir{-};
(0,7)*{\otimes};(5,0)*{\otimes}**\dir{-};
(-6,-3)*{\scriptstyle 1};(0,10)*{\scriptstyle 2};(6,-3)*{\scriptstyle 3};
(4,4)*{\scriptstyle \al};(0,-2)*{\scriptstyle -1-\al};
(0,8)*{};(0,-8)*{};
\endxy$$
\\\hline
\end{tabular}
\end{table}

 \subsection{ Quantum supergroups}
 Set $q_i :=q^{d_i}$, and then we have $$q_i^{a_{i,j}}=q_j^{a_{j,i}}$$ for all $i,j\in \mathbb{I}$. 
Denote the quantum integers and quantum binomial coefficients by  
\begin{align*}
[a]=\frac{q^a-q^{-a}}{q-q^{-1}}
, \qquad [a]_i =
    \frac{q_i^a-q_i^{-a}}{q_i-q_i^{-1}},\qquad
 \begin{bmatrix}a\\k\end{bmatrix}_i
 =\frac{[a]_i[a-1]_i\cdots[a-k+1]_i}{[k]_i!},
\end{align*}
for $a\in\Z$, $k\in\Z_{\geq 0}$, $i\in\mathbb{I}$.

Following \cite{Dri88,Ya94}, the (universal) quantum supergroup ${\ol{U}}_q(\mathfrak{g})$ is defined to be the $\C(q)$-superalgebra generated by $\genK_i$, $\genK'_i$, $\genE_i$, $\genF_i$, $i\in\mathbb{I}$, subject to the following relations:
    \begin{equation}\label{basic relation}
    \begin{aligned}
     &\genK_i\genK_j=\genK_j\genK_i,\qquad  
     \genK_i\genK'_j=\genK'_j\genK_i,\qquad \genK'_i\genK'_j=\genK'_j\genK'_i,\\   &\genK_i\genE_j=q_i^{a_{i,j}}\genE_j\genK_i,\quad \genK'_i\genE_j=q_i^{-a_{i,j}}\genE_j\genK'_i,\\     
     &\genK_i\genF_j=q_i^{-a_{i,j}}\genF_j\genK_i,\quad \genK'_i\genF_j=q_i^{a_{i,j}}\genF_j\genK'_i,     \\
     &\genE_i\genF_j-(-1)^{p(i)p(j)}\genF_j\genE_i=\delta_{ij}\frac{\genK_i-\genK'_i}{q_i-q_i^{-1}},
    \end{aligned}
     \end{equation}
    together with the higher order quantum Serre relations, see Proposition \ref{P:SerreRelations}.

Recall that the $q$-commutator on homogeneous elements $u\in {\ol{U}}_q(\mathfrak{g})_{\al},v\in{\ol{U}}_q(\mathfrak{g})_{\beta}$ is defined by
    \begin{align*}
        \ad_q u(v)=[u,v]_q=uv-(-1)^{p(u)p(v)}q^{(\al,\beta)}vu.
    \end{align*}

As shown in \cite{Ya94,Ya99}, higher order quantum Serre relations appear in the context of quantum supergroups associated with arbitrary Dynkin diagrams $\Gamma$. For later use, we adopt the notation  from \cite[Proposition 2.7]{CHW16}.

\begin{prp}  \cite[Proposition 2.7]{CHW16}
\label{P:SerreRelations}
The superalgebra ${\ol{U}}_q(\mathfrak{g})$  satisfies the following relations whenever the given Dynkin subdiagram $\Gamma$ appears:

\begin{enumerate}
\item[(Iso)] For $i,j\in \mathbb{I}_\one$ with $a_{i,j}=0$,
$$\genE_i\genE_j=-\genE_j\genE_i.$$
\item[(N-Iso)] For $i\in \mathbb{I}_\zero\cup \mathbb{I}_\niso$ and $i\neq j$,
$$ 
\ad_q \genE_i^{1-a_{i,j}}(\genE_j)=0.
$$
\item[(AB)] For
$$\hspace{.75in}\xy
(-10,0)*{\odot};(0,0)*{\otimes}**\dir{-};(0,0)*{\otimes};(10,0)*{\odot}**\dir{-};
(-10,-4)*{\scriptstyle i};(0,-4)*{\scriptstyle j};(10,-4)*{\scriptstyle k};
\endxy\quad(sgn_{ij}\neq sgn_{jk})$$
or
$$\xy
(-10,0)*{\textup\yy};(0,0)*{\otimes}**\dir{=};(0,0)*{\otimes};(10,0)*{\odot}**\dir{-};(-5,0)*{<};
(-10,-4)*{\scriptstyle i};(0,-4)*{\scriptstyle j};(10,-4)*{\scriptstyle k};
\endxy$$
$$
\ad_q \genE_j\cdot\ad_{q}\genE_k\cdot \ad_{q} \genE_j (\genE_i)=0.
$$

\item[(CD1)] For
$$\xy
(-10,0)*{\textup\fullmoon};(0,0)*{\otimes}**\dir{=};(0,0)*{\otimes};(10,0)*{\otimes}**\dir{-};(-5,0)*{>};
(-10,-4)*{\scriptstyle i};(0,-4)*{\scriptstyle j};(10,-4)*{\scriptstyle k};
\endxy$$
$$\ad_q \genE_j\cdot \ad_q (\ad_q \genE_j(\genE_k))\cdot \ad_q \genE_i\cdot \ad_q \genE_j (\genE_k)=0.$$
\item[(CD2)] For
$$\xy
(-10,0)*{\odot};(0,0)*{\textup\fullmoon}**\dir{-};(0,0)*{\textup\fullmoon};(10,0)*{\otimes}**\dir{-};(15,0)*{<};(10,0)*{\otimes};(20,0)*{\textup\fullmoon}**\dir{=};
(-10,-4)*{\scriptstyle i};(0,-4)*{\scriptstyle j};(10,-4)*{\scriptstyle k};(20,-4)*{\scriptstyle l};
\endxy$$
$$\ad_q \genE_k\cdot \ad_q \genE_j \cdot \ad_q \genE_k\cdot \ad_q \genE_l\cdot \ad_q \genE_k\cdot \ad_q \genE_j (\genE_i)=0.$$
\item[(D)] For
$$\xy
(0,0)*{\odot};(7,5)*{\otimes}**\dir{-};(0,0)*{\odot};(7,-5)*{\otimes}**\dir{-};(7,5)*{\otimes};(7,-5)*{\otimes}**\dir{=};
(-4,0)*{\scriptstyle i};(10,5)*{\scriptstyle j};(10,-5)*{\scriptstyle k};
\endxy$$
$$\ad_q \genE_k\cdot \ad_q \genE_j (\genE_i)=\ad_q \genE_j \cdot \ad_q \genE_k(\genE_i).$$
\item[(F1)] For
$$\xy
{\ar@3{-}(-15,0)*{\textup\fullmoon};(-5,0)*{\otimes}};(-10,0)*{>};
(-5,0)*{\otimes};(5,0)*{\textup\fullmoon}**\dir{=};(0,0)*{<};(5,0)*{\textup\fullmoon};(15,0)*{\textup\fullmoon}**\dir{-};
(-15,-4)*{\scriptstyle 1};(-5,-4)*{\scriptstyle 2};(5,-4)*{\scriptstyle 3};(15,-4)*{\scriptstyle 4};
\endxy$$
$$\ad_q E\cdot \ad_q E \cdot \ad_q \genE_2 \cdot \ad_q \genE_3 ( \genE_4)=0,$$
where $E=\ad_q(\ad_q \genE_1(\genE_2))\cdot \ad_q \genE_3(\genE_2)$.
\item[(F2)] For
$$\xy
{\ar@3{-}(-15,0)*{\textup\fullmoon};(-5,0)*{\otimes}};(-10,0)*{>};
(-5,0)*{\otimes};(5,0)*{\textup\fullmoon}**\dir{-};(10,0)*{<};(5,0)*{\textup\fullmoon};(15,0)*{\textup\fullmoon}**\dir{=};
(-15,-4)*{\scriptstyle 1};(-5,-4)*{\scriptstyle 2};(5,-4)*{\scriptstyle 3};(15,-4)*{\scriptstyle 4};
\endxy$$
\begin{align*}
&\ad_q(\ad_q \genE_1(\genE_2))  \cdot \ad_q (\ad_q \genE_3(\genE_2)) \cdot \ad_q \genE_3(\genE_4)
\\
&= \ad_q(\ad_q \genE_3(\genE_2))\cdot \ad_q (\ad_q \genE_1(\genE_2)) \cdot \ad_q \genE_3(\genE_4).
\end{align*}

\item[(F3)] For
$$\xy
(-10,0)*{\otimes};(0,0)*{\otimes}**\dir{=};(0,0)*{\otimes};(10,0)*{\textup\fullmoon}**\dir{=};(5,0)*{<};
(-10,-4)*{\scriptstyle 1};(0,-4)*{\scriptstyle 3};(10,-4)*{\scriptstyle 4};
\endxy$$
$$\ad_q \genE_3\cdot\ad_{q}\genE_1\cdot \ad_{q} \genE_3 (\genE_4)=0.$$
\item[(F4)]  For
$$\xy
{\ar@3{-}(-5,0)*{\otimes};(0,7)*{\otimes}};(-5,0)*{\otimes};(5,0)*{\otimes}**\dir{=};(0,7)*{\otimes};(5,0)*{\otimes}**\dir{-};
(-8,0)*{\scriptstyle 1};(0,10)*{\scriptstyle 2};(8,0)*{\scriptstyle 3};
\endxy$$
$$[3]\ad_q \genE_1\cdot \ad_q \genE_2 (\genE_3)+[2]\ad_q \genE_2 \cdot \ad_q \genE_1(\genE_3)=0.$$
\item[(G1)] For
$$ \xy
{\ar@3{-}(0,0)*{\otimes};(10,0)*{\textup\fullmoon}};(5,0)*{<};(-10,0)*{\otimes};(0,0)*{\otimes}**\dir{-};
(-10,-4)*{\scriptstyle 1};(0,-4)*{\scriptstyle 2};(10,-4)*{\scriptstyle 3};
\endxy$$
$$\ad_q E\cdot \ad_q E \cdot \ad_q E \cdot \ad_q \genE_2(\genE_1)=0,$$  	
where $E=\ad_q \genE_2(\genE_3)$.

\item[(G2)] For
$$ \xy
{\ar@3{-}(0,0)*{\otimes};(10,0)*{\textup\fullmoon}};(5,0)*{<};(-10,0)*{\textup\newmoon};(0,0)*{\otimes}**\dir{=};
(-10,-4)*{\scriptstyle 1};(0,-4)*{\scriptstyle 2};(10,-4)*{\scriptstyle 3};
\endxy$$
$$\ad_q \genE_2\cdot \ad_q \genE_3 \cdot \ad_q \genE_3\cdot \ad_q \genE_2(\genE_1)=
  	\ad_q \genE_3\cdot \ad_q \genE_2 \cdot \ad_q \genE_3\cdot \ad_q \genE_2(\genE_1).$$

\item[(G3)] For
$$\xy
{\ar@3{-}(-5,0)*{\otimes};(5,0)*{\otimes}};(-5,0)*{\otimes};(0,7)*{\textup\fullmoon}**\dir{=};(0,7)*{\textup\fullmoon};(5,0)*{\otimes}**\dir{-};
(-8,0)*{\scriptstyle 1};(0,10)*{\scriptstyle 2};(8,0)*{\scriptstyle 3};
\endxy$$
$$\ad_q \genE_1\cdot \ad_q \genE_2 (\genE_3)-[2]\ad_q \genE_2 \cdot \ad_q \genE_1(\genE_3)=0.$$

\item[($D\al$)]  For
$$\xy
(-5,0)*{\otimes};(0,7)*{\otimes}**\dir{-};
(-5,0)*{\otimes};(5,0)*{\otimes}**\dir{-};
(0,7)*{\otimes};(5,0)*{\otimes}**\dir{-};
(-8,0)*{\scriptstyle 1};(0,10)*{\scriptstyle 2};(8,0)*{\scriptstyle 3};
(4,4)*{\scriptstyle \al};(0,-2)*{\scriptstyle -1-\al};
\endxy$$  				
$$[\alpha+1]\ad_q \genE_1\cdot \ad_q \genE_3 (\genE_2)+[\alpha]\ad_q \genE_3 \cdot \ad_q \genE_1(\genE_2)=0.$$
\end{enumerate}
\end{prp}

\begin{rmk}
By \cite[Proposition 8]{Gee06}, the quantum supergroup ${\ol{U}}_q(\mathfrak{g})$  has the structure of a Hopf superalgebra, with coassociative comultiplication ${\ol{\Delta}}$, counit $\ol{\ep}$ and antipode $\ol{S}$ respectively given by 
\begin{align*}
    &{\ol{\Delta}}(\genE_i)=\genE_i\otimes 1+\genK_i\otimes\genE_i, 
    &\ol{\ep}(\genE_i)=0,\qquad
    &\ol{S}(\genE_i)=-\genK_i^{-1}\genE_i,
    \\
    &{\ol{\Delta}}(\genF_i)=1\otimes \genF_i+\genF_i\otimes\genK'_i,
    &\ol{\ep}(\genF_i)=0,\qquad
    &\ol{S}(\genF_i)=-\genF_i\genK_i^{'-1},\\
    &{\ol{\Delta}}(\genK_i)=\genK_i\otimes\genK_i,
    &\ol{\ep}(\genK_i)=1,\qquad 
    &\ol{S}(\genK_i)=\genK_i^{-1},
    \\    
    &{\ol{\Delta}}(\genK'_i)=\genK'_i\otimes\genK'_i,
    &\ol{\ep}(\genK'_i)=1,\qquad 
    &\ol{S}(\genK'_i)=\genK_i^{'-1}.
\end{align*}
\end{rmk}

Recall from \cite[Section 1.9]{Ya94} and \cite[Section 3.1]{SW25} that the quantum supergroup ${\ol{\mathbf{U}}}:={\ol{U}}_q(\mathfrak{g})\oplus{\ol{U}}_q(\mathfrak{g})\varrho$ is a  Hopf algebra, 
where $\varrho$  denotes an algebra involution of parity $0$ defined on ${\ol{U}}_q(\mathfrak{g})$  by 
\begin{align}
\label{varrho'}
\begin{split}
\varrho(\genK_i)=  \genK_i,\quad 
  \varrho(\genK'_i)&=  \genK'_i,\quad
  \varrho(\genE_i)= (-1)^{p(i)} \genE_i,\quad
\varrho(\genF_i)= (-1)^{p(i)} \genF_i,\quad \mbox{ for } i\in \mathbb{I};\\
\varrho^2&=1,\quad x\cdot \varrho=\varrho\cdot \varrho(x),  \mbox{ for } x\in {\ol{U}}_q(\mathfrak{g}).
\end{split}
  \end{align}
According to \cite[Proposition 1.9.1]{Ya94}, in this paper we consider the Hopf algebra structure on ${\ol{\mathbf{U}}}$ given by  the comultiplication $\Delta$, counit $\ep$ and antipode $S$ defined as follows.
\begin{equation}\label{hopfstructure}
\begin{aligned}
  &\Delta(\genE_i)=E_i\otimes\varrho^{p(i)} +\genK_i\otimes \genE_i, &\ep(\genE_i)=0,\qquad &S(\genE_i)=-\genK_i^{-1}\genE_i\varrho^{p(i)};\\
    &\Delta(\genF_i)=\genF_i\otimes \varrho^{p(i)}\genK'_i+1\otimes \genF_i, &\ep(\genF_i)=0,\qquad
    &S(\genF_i)=-\genF_i\varrho^{p(i)}\genK_i^{'-1};\\
    &\Delta(\genK_i)=\genK_i\otimes \genK_i, &\ep(\genK_i)=1,\qquad &S(\genK_i)=\genK^{-1}_{i};\\
    &\Delta(\genK'_i)=\genK'_i\otimes \genK'_i, &\ep(\genK'_i)=1,\qquad &S(\genK'_i)=\genK^{'-1}_{i};\\
    &\Delta(\varrho)=\varrho\otimes \varrho,&\ep(\varrho)=1,\qquad &S(\varrho)=\varrho.
\end{aligned}
\end{equation}

\begin{rmk}
It is straightforward to verify that ${\ol{\mathbf{U}}}$ is a Hopf algebra with the comultiplication $\Delta$, counit $\ep$ and antipode $S$, which are different from those in \cite[(3.4)]{SW25}.
\end{rmk}

For the universal quantum supergroup 
\begin{align}
\label{defU}
    {\mathbf{U}}:={\ol{U}}_q(\mathfrak{g})\oplus{\ol{U}}_q(\mathfrak{g})\varrho\oplus{\ol{U}}_q(\mathfrak{g})\varrho'\oplus{\ol{U}}_q(\mathfrak{g})\varrho\varrho',
\end{align}
where $\varrho'$ is a second parity-0 involution analogous to $\varrho$, satisfying relations \eqref{varrho'} as $\varrho$ and 
\begin{align}
    \varrho\varrho'=\varrho'\varrho.\label{varrho''}
\end{align}
One verifies that ${\mathbf{U}}$ carries a Hopf algebra structure with comultiplication, counit and antipode given by \eqref{hopfstructure} together with
\begin{align*}
    \Delta(\varrho')=\varrho'\otimes\varrho',\quad
    \ep(\varrho')=1,\quad
    S(\varrho')=\varrho'.
\end{align*}

The quantum supergroup ${{\mathbf{U}}}$ is $\Z\Phi$-graded. The grading is given by 
\begin{align*}
    \mbox{deg}\genK_h=0=\mbox{deg}\genK'_h, \quad \mbox{deg}\varrho=0=\mbox{deg}\varrho',    
    \quad\mbox{deg}\genE_i=\alpha_i,\quad \mbox{deg}\genF_i=-\alpha_i,
\end{align*}
for all $h\in \Z\Phi$ and $i\in\mathbb{I}$.


\begin{rmk}
    The Drinfeld-Jimbo quantum supergroup ${\hat{\mathbf{U}}}:={\hat{\mathbf{U}}}_q(\mathfrak{g})$ is defined as the $\C(q)$-superalgebra generated by $\genK_i$, $\genK^{-1}_i$, $\varrho$, $\genE_i$, $\genF_i$ $(i\in\mathbb{I})$, subject to the relations deduced from  \eqref{basic relation}, \eqref{varrho'}  and \eqref{varrho''} with  $\genK'_i$ replaced by $\genK^{-1}_i$ and $\varrho'$ replaced by $\varrho^{-1}$. We can also view ${\hat{\mathbf{U}}}$ as the quotient superalgebra of ${{\mathbf{U}}}$ modulo the ideal generated by $\varrho\varrho'-1$  and  $\genK_i\genK'_i-1$ ($i\in\mathbb{I}$); see \cite{Dri88}.
\end{rmk}

Let ${\hat{\mathbf{U}}}_q^0$ (resp. ${{\mathbf{U}}}_q^0$) be the sub-superalgebra of ${\hat{\mathbf{U}}}$ (resp. $\mathbf{U}$) generated by $\genK_i, \varrho$ (resp. $\genK_i$, $\genK_i'$, $ \varrho$, $ \varrho'$), and denote by ${\hat{\mathbf{U}}}_q^+$ (resp. ${{\mathbf{U}}}_q^+$) the sub-superalgebra of ${\hat{\mathbf{U}}}$ (resp. $\mathbf{U}$) generated by $\genE_i$, $i\in\mathbb{I}$. Let ${\hat{\mathbf{U}}}_q^-$ (resp. ${{\mathbf{U}}}_q^-$) be the sub-superalgebra of ${\hat{\mathbf{U}}}$ (resp. $\mathbf{U}$) generated by $\genF_i$ $(i\in\mathbb{I})$.
Let ${\hat{\mathbf{U}}}_q^{\geq0}$ (resp. ${\hat{\mathbf{U}}}_q^{\leq0}$) be the sub-superalgebra of ${\hat{\mathbf{U}}}$ generated by $\genE_i$, $\genK_i$, $\varrho$ (resp. $\genF_i$, $\genK_i$, $\varrho$). It is well-known that  the multiplication maps induce isomorphisms 
\begin{align*}
  {\hat{\mathbf{U}}}_q^+\otimes{\hat{\mathbf{U}}}_q^0\cong {\hat{\mathbf{U}}}_q^{\geq0},\qquad 
  {\hat{\mathbf{U}}}_q^-\otimes{\hat{\mathbf{U}}}_q^0\cong {\hat{\mathbf{U}}}_q^{\leq0}.
\end{align*} 
Moreover, the multiplication maps 
\begin{align*}
  {\hat{\mathbf{U}}}_q^-\otimes{\hat{\mathbf{U}}}_q^0\otimes{\hat{\mathbf{U}}}_q^+\cong {\hat{\mathbf{U}}},\qquad  {{\mathbf{U}}}_q^-\otimes{{\mathbf{U}}}_q^0\otimes{{\mathbf{U}}}_q^+\cong {{\mathbf{U}}}
\end{align*}
are  isomorphisms of super vector spaces.

\subsection{{Quantum supersymmetric pairs}}In this paper, we only consider quantum supersymmetric pairs associated with (quasi-split) Satake diagrams containing no black nodes. 

For a Cartan matrix $A=(a_{i,j})_{i,j\in\mathbb{I}}$, let Inv$(A)$ be the group of permutations $\tau$ of the set $\mathbb{I}$ such that $d_i a_{i,j}=d_{\tau i} a_{\tau i, \tau j}$, for all $i,j\in\mathbb{I}$, and $\tau^2=$Id. Then $\tau\in $Inv$(A)$ can be viewed as an involution (which is allowed to be the identity) of the corresponding Dynkin diagram. 

 Following \cite[(3.10)-(3.11)]{SW25}, we have
\begin{align*}
    q_i=\begin{cases}
        q_{\tau i}, &\mbox{ if }~i=\tau i\\
        q_{\tau i}, &\mbox{ if }~i\neq\tau i, i\in \mathbb{I}_\zero\\
        q_{\tau i}^{-1}, &\mbox{ if }~i\neq\tau i, i\in \mathbb{I}_\one\\
    \end{cases}
\end{align*}
and the permutation $\tau$ of $\mathbb{I}$ gives rise to an algebra automorphism of  ${{\mathbf{U}}}$, defined by 
\begin{equation}\label{tau}
    \begin{aligned}
        &\tau(\genE_i)=\begin{cases}
\genE_{i},&\text{ if }\tau i=i\\
(-1)^{p(i)}\genE_{\tau i},&\text{ if }\tau i\neq i,
\end{cases}\\
&\tau(\genF_i)=\genF_{\tau i}, \quad \tau(\genK_i)=\genK_{\tau i}, \quad \tau(\genK'_i)=\genK'_{\tau i},\quad
\tau(\varrho)=\varrho,\quad 
\tau(\varrho')=\varrho',\quad
\mbox{ for all } i\in \mathbb{I}.
    \end{aligned}
\end{equation}

We recall the definition of the quasi-split iquantum supergroup from \cite[Section 4]{SW25}.

\begin{dfn}\label{def:Ui}    
The (quasi-split) iquantum supergroup associated with the super Satake diagram $(\mathbb{I},\tau)$, denoted by $\Ui$, with parameters $\varsigma_i \in \C(q)^* \text{ for }i\in \mathbb{I},$    is the $\C(q)$-subalgebra of {${{\mathbf{U}}}$} generated by    
\begin{align*}        
B_i= \genF_i+\varsigma_i\genE_{\tau i}\genK'_i,\quad \tilde{k}_i=\genK_{i}\genK'_{\tau i}, \quad \tilde{\varrho}=\varrho\varrho'.    
\end{align*}
\end{dfn}
Here the parameters are assumed to satisfy:
\begin{align}
  &\varsigma_i=\varsigma_{\tau i}~ \mbox{ if } a_{i,\tau i}=0 ~\mbox{ and } \tau (\al_i)=\al_{\tau i};\\
  &\varsigma_i=\varsigma_{\tau i}~ \mbox{ if } a_{i,\tau i}\in 2\Z ~\mbox{ and } p(i)=1.
\end{align}

The algebra $\Ui$ is a right coideal subalgebra of ${{\mathbf{U}}}$, i.e., $\Delta(\Ui)\subset \Ui\otimes {\mathbf{U}}$ (see \cite[Proposition 4.2]{SW25}). The pair $({{\mathbf{U}}}, \Ui)$ is called a quantum supersymmetric pair associated with $(\mathbb{I},\tau)$. If $\tau=$Id, the pair is said to be $split$.

\subsection{\texorpdfstring{$\mathrm i$}{i}Hopf algebras}
\label{s2-4}Assume that $H=(H,m,\eta,\Delta,\ep,S)$ is a Hopf algebra with multiplication $m$, unit $\eta$, comultiplication $\Delta$, counit $\ep$ and antipode $S$. We assume throughout that the antipode $S$ is invertible. We write $ab=m(a,b)$ and use the Sweedler notation $\Delta(a)=\sum a_{(1)}\otimes a_{(2)}$ for $a,b\in H$. Define  $\Delta^{(2)}:=(\Delta\otimes\Id)\cdot\Delta $, and then $\Delta^{(l)}$ by iteration. Denote 
	$\Delta^{(l)}(a)=\sum a_{(1)}\otimes a_{(2)}\otimes\cdots\otimes a_{(l+1)}$. 


    Let $H,H'$ be two $\C(q)$-Hopf algebras with the comultiplication $\Delta_H,\Delta_{H'}$, counit $\ep_H,\ep_{H'}$ and antipode $S_H,S_{H'}$, respectively. Let $\psi(-,-):H{\otimes}{H'}\to \C(q)$ be a $\C(q)$-bilinear form. Recall from \cite[1.4]{Ya94} that $\psi(-,-)$ is a Hopf pairing if $\psi(-,-)$ satisfies the following relations:
    \begin{align}
    \label{hopfpairing}
    \begin{split}
        &(1)\quad \psi( a_{1} a_2,b)=\psi( a_1\otimes a_2,\Delta_{{H'}}(b)),\quad        
        \psi( a,b_1 b_2)=\psi(\Delta_{H}(a),b_1\otimes b_2),\\        
       &(2)\quad \psi (a,S_{H'}(b))=\psi (S_H(a),b),\\        
       &(3)\quad \psi(\eta_H(1),b)=\ep_{H'}(b),\quad \psi( a,\eta_{H'}(1))=\ep_H(a),    
    \end{split}
        \end{align}
    where $a,a_1,a_2\in H$ and $b,b_1,b_2\in {H'}$.
    A Hopf pairing of $H$ and itself will  simply be called a Hopf pairing of $H$.

Following \cite{CLPRW25}, for a Hopf algebra $H$ with a Hopf pairing $\psi$, we can define the associated iHopf algebra $H^\imath$, which is the same vector space as $H$, equipped with a new multiplication: 
\begin{equation}
    \begin{aligned}
     a\circ b=\sum \psi( b_{(2)},a_{(1)}  ) a_{(2)} b_{(1)}.
    \end{aligned}
\end{equation}

Let $H\otimes H$ be the tensor product Hopf algebra endowed with 
\begin{align*}
    &(a\otimes b)(c\otimes d)=ac\otimes bd,\\
    &\Delta(a\otimes b)=\sum a_{(1)}\otimes b_{(1)}\otimes a_{(2)}\otimes b_{(2)},\\
    &\ep(a\otimes b)=\ep(a)\ep(b), ~\mbox{ for any }~ a,b\in H.
\end{align*}
Define a (twisted) Hopf pairing $\psi_{*}$ of $H\otimes H$ by 
\begin{align*}
    \psi_{*}(c\otimes d, a\otimes b)=\psi(a,d)\psi(c,b).
\end{align*}
Then there is an iHopf algebra $(H\otimes H)^{\imath}$ associated with $(H\otimes H,\psi_{*})$, called the iHopf algebra of diagonal type. Explicitly, the multiplication is given by 
\begin{align*}
    (a\otimes b)\circ(c\otimes d)=\sum \psi(a_{(1)},d_{(2)})\cdot \psi(c_{(2)},b_{(1)})\cdot a_{(2)}c_{(1)}\otimes b_{(2)}d_{(1)},~\mbox{for any}~a,b,c,d\in H.
\end{align*}
Moreover, $(H\otimes H)^{\imath}$ is itself a Hopf algebra.
\begin{lem}\cite[Proposition 2.7]{CLPRW25}
The algebra $((H\otimes H)^{\imath},\circ,1)$ is a Hopf algebra with comultiplication, counit and antipode given by
\begin{align*}
    &\Delta^{\imath}(a\otimes b)=\psi(a_{(2)},b_{(2)})\cdot(a_{(1)}\otimes b_{(3)})\otimes(a_{(3)}\otimes b_{(1)}),\\
    &\ep^{\imath}(a\otimes b)=\psi(a,S^{-1}(b)),\\
    &S^{\imath}(a\otimes b)= S(a)\otimes S^{-1}(b).
\end{align*}
\end{lem}
 
\section{Hopf pairing on the Borel quantum supergroup ${\hat{\mathbf{U}}}_q^{\geq0}$}\label{Hopfpair}

In this section, we study the Hopf pairing on the Borel quantum supergroup. 
The results of this section, as well as those of Section 4, apply to both basic and affine quantum supergroups.

Let ${\dot{\mathbf{U}}}_q^{\geq0}$ be a $\C(q)$-superalgebra generated by $\genE_i,\genK_i,\varrho$ with relations \eqref{basic relation}-\eqref{varrho'}. Denote by $\Theta(\Gamma)\subset {\dot{\mathbf{U}}}_q^{\geq0}$ the set of all elements corresponding to the higher quantum Serre relations given in Proposition \ref{P:SerreRelations}.

\begin{lem}\label{Delta}
    Let $\mathcal{F}$ be the ideal of ${\dot{\mathbf{U}}}_q^{\geq0}$ generated  by the elements in $\Theta(\Gamma)$. Then for any $\theta\in\Theta(\Gamma)$
    with $\beta=\mathrm{deg}(\theta)$, we have
    $$\Delta(\theta)=\theta\otimes\varrho^{p(\beta)}+\genK_\beta\otimes\theta.$$ Consequently,  $$\Delta(\mathcal{F})\subset \mathcal{F}\otimes{\dot{\mathbf{U}}}_q^{\geq0}+{\dot{\mathbf{U}}}_q^{\geq0}\otimes\mathcal{F}.$$
\end{lem}
\begin{proof}
   From \cite[Proposition~4.3.1]{Ya94}, we know that the lemma holds for types A-D. It remains to verify that it also holds for the exceptional cases $F(3|1)$, $G(3)$, and $D(2|1;\al)$.
   Here we consider the case of $F(3|1)$ given by
$$
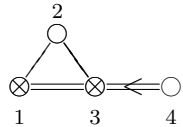
\xy
(-5,0)*{\otimes};(0,7)*{\fullmoon}**\dir{-};(-5,0)*{\otimes};(5,0)*{\otimes}**\dir{=};(0,7)*{\fullmoon}**\dir{-};(5,0)*{\otimes}**\dir{-};(15,0)*{\fullmoon}**\dir{=};(10,0)*{<};
(-5,-4)*{\scriptstyle 1};(0,10)*{\scriptstyle 2};(5,-4)*{\scriptstyle 3};(15,-4)*{\scriptstyle 4};
\endxy$$
The other cases can be proved similarly.

In this case, the ideal $\mathcal{F}$ is generated by the elements 
\begin{align*}
    &s_{2\al_1}=\genE_1^2,\quad s_{2\al_3}=\genE_3^2,\quad  s_{\al_1+2\al_2}=\ad_q \genE_2\cdot \ad_q \genE_2(\genE_1),\\
    &s_{\al_3+2\al_2}=\ad_q \genE_2\cdot \ad_q \genE_2(\genE_3),\quad
    s_{\al_3+2\al_4}=\ad_q \genE_4\cdot \ad_q \genE_4(\genE_3),\\    &s_{\al_1+\al_2+\al_3}=\ad_q \genE_1\cdot \ad_q \genE_3 (\genE_2)-\ad_q \genE_3\cdot \ad_q \genE_1(\genE_2),
\end{align*}
and 
\begin{align*}
    t_{3134}=\ad_q \genE_3\cdot\ad_{q}\genE_1\cdot \ad_{q} \genE_3 (\genE_4).
\end{align*}
By \cite[Proposition~4.3.1]{Ya94}, for any $s_\beta$ we have
\begin{align*}
\Delta(s_\beta)=s_\beta\otimes\varrho^{p(\beta)}+\genK_\beta\otimes s_\beta.
\end{align*}
It therefore remains to compute $\Delta(t_{3134})$. By a direct computation, we obtain 
\begin{align*}
    \Delta(\ad_{q} \genE_3 (\genE_4))= \ad_{q} \genE_3 (\genE_4)\otimes\varrho+\genK_3\genK_4\otimes \ad_{q} \genE_3 (\genE_4)+(q_3^{-a_{3,4}}-q_3^{a_{3,4}})\genK_4\genE_3\otimes\varrho\genE_4
\end{align*}
and
\begin{align*}
    \Delta(\ad_q \genE_1\cdot\ad_{q} \genE_3 (\genE_4))=&\big(\ad_q \genE_1\cdot\ad_{q} \genE_3 (\genE_4)\big)\otimes1\\
    &-(q_1^{-a_{1,3}}-q_1^{a_{1,3}})\genK_3\genK_4\genE_1\otimes\ad_q\genE_3(\genE_4)\varrho\\
    &+(q_3^{-a_{3,4}}
    -q_3^{a_{3,4}})\genK_4\ad_q\genE_1(\genE_3)\otimes\genE_4\\
    &+\genK_1\genK_3\genK_4\otimes \big(\ad_q \genE_1\cdot\ad_{q} \genE_3 (\genE_4)\big).
\end{align*}
Let $\mathcal{F}'$ be the ideal generated by $s_{2\al_3}$.
We have $\Delta(\mathcal{F}')\subset \mathcal{F}'\otimes{\dot{\mathbf{U}}}_q^{\geq0}+{\dot{\mathbf{U}}}_q^{\geq0}\otimes\mathcal{F}'$ and
$$\ad_q \genE_3\cdot\ad_q\genE_3(\genE_4)\equiv\ad_q \genE_3\cdot\ad_q\genE_3(\genE_1)\equiv 0~ (\mbox{modulo} ~\mathcal{F}').$$
Note that $q_1^{a_{1,3}}=q_3^{a_{3,1}}=q_3^{-a_{3,4}}$.
We obtain 
\begin{align*}
\Delta(t_{3134})
&=t_{3134}\otimes\varrho+\genK_1\genK_3^2\genK_4\otimes t_{3134}\\
&\quad-(q_3^{-a_{3,4}}-q_3^{a_{3,4}})\genK_4\cdot\ad_q(\ad_q\genE_1(\genE_3))(\genE_3)\otimes\genE_4\varrho\\
&\quad+(q_1^{-a_{1,3}}-q_1^{a_{1,3}})\genK_3\genK_4\cdot\ad_q\genE_1(\genE_3)\otimes\ad_q\genE_3(\genE_4)\\
&\quad+(q_3^{-a_{3,4}}-q_3^{a_{3,4}})\genK_3\genK_4\cdot\ad_q\genE_1(\genE_3)\otimes\ad_q\genE_3(\genE_4)\\
&\quad-(q_1^{-a_{1,3}}-q_1^{a_{1,3}})\genK_3^2\genK_4\genE_1\otimes\ad_q\genE_3\cdot \ad_q\genE_3(\genE_4)\varrho\\
&\equiv t_{3134}\otimes\varrho+\genK_1\genK_3^2\genK_4\otimes t_{3134}\quad (\mbox{modulo}~ \Delta(\mathcal{F}'))\\
&=t_{3134}\otimes\varrho^{p(\al_1+2\al_3+\al_4)}+\genK_{\al_1+2\al_3+\al_4}\otimes t_{3134}.
\end{align*}
The result now follows from $\deg(t_{3134})=\al_1+2\al_3+\al_4$.
\end{proof}


Note that ${\hat{\mathbf{U}}}_q^{\geq0}\cong{\dot{\mathbf{U}}}_q^{\geq0}/\mathcal{F} $. Using the pairing given in \cite[Proposition 4.3]{LWY22}, we obtain a pairing on the Borel quantum supergroup ${\hat{\mathbf{U}}}_q^{\geq0}$ as follows.
\begin{prp}\label{pairing}
   {There is a unique non-degenerate Hopf pairing on ${\hat{\mathbf{U}}}_q^{\geq0}$ with}
    \begin{align}
    \label{hopfpair}
    \begin{split}
       &\langle\genK_i,\genK_j\rangle =q_i^{a_{i,j}}, \quad\langle\genE_i,\genE_j\rangle=-\frac{\delta_{i,j}}{q_i-q_i^{-1}},\quad \langle\genK_i,\genE_j\rangle=0,\\
       & {\langle \varrho, \varrho \rangle=-1}, \quad \langle \varrho,\genK_i\rangle =1,\quad \langle\varrho,\genE_i\rangle=0.  
    \end{split}
    \end{align}
\end{prp}
\begin{proof}

First, we show that the pairing $\langle-,-\rangle$ defined by \eqref{hopfpair} is well-defined on ${\hat{\mathbf{U}}}_q^{\geq0}$, which is equivalent to proving that the ideal generated by the relations of $E_i,K_i,\varrho$ in \eqref{basic relation}, \eqref{varrho'} and the higher order quantum Serre relations in Proposition \ref{P:SerreRelations} belongs to the radical of $\langle-,-\rangle$.

We shall prove $\langle a\theta a',b\rangle=0$ for each generating relation $\theta$ and $a,a',b\in {\hat{\mathbf{U}}}_q^{\geq0}$.
Note that\[
\langle a\theta a',b\rangle=\sum\langle a,b_{(1)}\rangle\langle \theta ,b_{(2)}\rangle\langle  a',b_{(3)}\rangle.
\]
Hence it suffices to prove $\langle \theta,b\rangle=0$ for any $b\in {\hat{\mathbf{U}}}_q^{\geq0}$. Since $b$ can be expressed as a sum of elements with the form $b_1b_2\cdots b_k$ for $k\in\mathbb{N}$ and $b_j\in\{\genE_i,\genK_i,\varrho\ |\ i\in \mathbb{I}\}$, we have
\[\langle \theta,b\rangle=
\langle \theta,b_1\cdots b_k\rangle=\langle\theta_{(1)},b_1\rangle\cdots\langle\theta_{(k)},b_k\rangle.
\]

We now consider the case where $\theta$ is taken to be one of the relations involving $E_i,K_i$ from \eqref{basic relation}, or one of the higher-order quantum Serre relations. Denote by ${\ol{U}}_q(\mathfrak{g})^{\ge0}$ the subalgebra of ${\hat{\mathbf{U}}}_q^{\geq0}$ generated by $K_i,E_i(i\in\mathbb{I})$. Hence, by \cite[Proposition 4.3]{LWY22}, we obtain a {Hopf pairing} of ${\ol{U}}_q(\mathfrak{g})^{\ge0}$ given by
\begin{align*}
\langle\genK_i,\genK_j\rangle =q_i^{a_{i,j}}, \quad\langle\genE_i,\genE_j\rangle=-\frac{\delta_{i,j}}{q_i-q_i^{-1}},\quad \langle\genK_i,\genE_j\rangle=0.
\end{align*}
It remains to verify the case that $\varrho$ occurs in $b$. 

{If $\theta\in\{\genK_i \genE_j-q_i^{a_{ij}}\genE_j \genK_i,\genK_i \genK_j-\genK_j \genK_i\}$}, we have the equations
  \begin{align*}
       \langle \genK_i \genE_j, \varrho  \rangle
      =0
      =\langle q_i^{a_{ij}}\genE_j \genK_i, \varrho  \rangle,\qquad\langle \genK_i \genK_j, \varrho  \rangle
      =1
      =\langle \genK_j \genK_i, \varrho  \rangle.
  \end{align*}
 The same argument then yields $\langle a\theta a',b\rangle=0$.

If $\theta\in \Theta(\Gamma)$, then by Lemma \ref{Delta} we see that
\begin{align*}   
\Delta(\theta)=\theta\otimes\varrho^{p(\beta)}+\genK_\beta\otimes\theta,
\end{align*} 
where $\beta=\mbox{deg}(\theta)$. Then for $\Delta^{(l)}(\theta)=\sum\theta_{(1)}\otimes{\cdots\otimes \theta_{(l+1)}}$, there must be one $\theta_{(s)}=\theta$ for each term being summed.
By a direct computation, it follows
    \begin{align*}
     &\langle \prod\limits_{u=1}^{l}\genE_{i_u}, \genK_{i_0} \rangle=0,\quad\langle \prod\limits_{u=1}^{l}\genE_{i_u}, \varrho \rangle=0,\\
        &\langle \prod\limits_{u=1}^{l}\genE_{i_u}, \genE_{i_0} \rangle
        =\langle \prod\limits_{u=1}^{l-1}\genE_{i_u}, \genK_{i_0} \rangle\langle  \genE_{i_{l}}, \genE_{i_0} \rangle=0,
    \end{align*}
 for any $l>1,i_0,i_1,\cdots i_l\in\mathbb{I}$. 
 
 {If $\theta\in\{\genK_i \varrho-\varrho\genK_i, \genE_i \varrho-\varrho\cdot\varrho(\genE_i),\varrho^2-1\} $,}
in each tensor summand of $\Delta^{(k)}(\theta)$,
there is a unique tensor factor $\theta_{(i)}$ equal to $\theta$. A direct calculation shows that
    \begin{align*}
    &\langle \genK_i \varrho, \genE_j  \rangle
      =0
      =\langle \varrho\genK_i , \genE_j  \rangle,
      \qquad\langle \genK_i \varrho, \genK_j  \rangle
      =\langle \genK_i , \genK_j  \rangle
      =\langle \varrho\genK_i , \genK_j  \rangle,
      \\&\langle \genK_i \varrho, \varrho  \rangle
      =\langle \varrho , \varrho  \rangle
      =\langle \varrho\genK_i , \varrho  \rangle;\\
      &\langle \genE_i \varrho, \genE_i  \rangle
      {=\langle \genE_i, \genE_i  \rangle\langle\varrho,\varrho^{p(i)}\rangle
      =(-1)^{p(i)}\langle \genE_i, \genE_i  \rangle=\langle\varrho,\genK_i\rangle\langle \varrho(\genE_i), \genE_i  \rangle}
      =\langle \varrho\cdot\varrho(\genE_i) , \genE_i  \rangle,\\
      &\langle \genE_i \varrho, \genE_j  \rangle
      {=\langle \genE_i, \genE_j  \rangle\langle\varrho,\varrho^{p(j)}\rangle
      =0=\langle\varrho,\genK_j\rangle\langle \varrho(\genE_i), \genE_j  \rangle}
      =\langle \varrho\cdot\varrho(\genE_i) , \genE_j  \rangle~ \mbox{for}~ i\neq j,\\
      &\langle \genE_i \varrho, \genK_j  \rangle
      =0
      =\langle \varrho\cdot\varrho(\genE_i) , \genK_j  \rangle,\qquad\langle \genE_i \varrho, \varrho  \rangle
      =0
      =\langle \varrho\cdot\varrho(\genE_i) , \varrho  \rangle;
       \\&\langle \varrho \varrho, \genE_i  \rangle
      =0
      =\langle 1 , \genE_i  \rangle,
      \qquad\langle \varrho \varrho, \genK_i  \rangle
      =1
      =\langle 1, \genK_i  \rangle,\qquad
      \langle \varrho \varrho, \varrho  \rangle
      =1
      =\langle 1, \varrho  \rangle.    
\end{align*}
It follows that $\langle\theta_{(i)},b_i\rangle=0$.  
 Hence, $\theta$ belongs to the radical of $\langle-,-\rangle$.
 
     Next, we verify the relations (2)-(3) given in \eqref{hopfpairing}. A check on the generators is sufficient.
    \begin{align*}
     &\langle \genE_i, S(\genE_j)\rangle
     =\langle\genE_i,-\genK_j^{-1}\genE_j\varrho^{p(j)}\rangle
     =\langle-\genK_i^{-1}\genE_i\varrho^{p(i)},\genE_j\rangle
     =\langle S(\genE_i),\genE_j\rangle,\\
     &\langle \genE_i,S(\genK_j)\rangle
     =\langle\genE_i,\genK_j^{-1}\rangle
     =0
     =\langle-\genK_i^{-1}\genE_i\varrho^{p(i)},\genK_j\rangle
     =\langle S(\genE_i),\genK_j\rangle,\\
     &\langle \genE_i,S(\varrho)\rangle
     =\langle\genE_i,\varrho\rangle
     =0
     =\langle-\genK_i^{-1}\genE_i\varrho^{p(i)},\varrho\rangle
     =\langle S(\genE_i),\varrho\rangle,\\
     &\langle \genK_i,S(\genK_j)\rangle
     =\langle\genK_i,\genK_j^{-1}\rangle
     =q_i^{-a_{i,j}}
     =\langle\genK_i^{-1},\genK_j\rangle
     =\langle S(\genK_i),\genK_j\rangle,\\
      &\langle \genK_i,S(\varrho)\rangle
     =\langle\genK_i,\varrho\rangle
     =1
     =\langle\genK_i^{-1},\varrho\rangle
     =\langle S(\genK_i),\varrho\rangle,\\
 &\langle \varrho,S(\varrho)\rangle
     =\langle\varrho,\varrho\rangle
     =\langle S(\varrho),\varrho\rangle;\\
     &\langle 1,\genE_i\rangle
     =\langle \varrho^2,\genE_i\rangle=0=\ep(\genE_i),\\
     &\langle 1,\genK_i\rangle
     =\langle \varrho^2,\genK_i\rangle=1=\ep(\genK_i),\\
     &\langle 1,\varrho\rangle
     =\langle \varrho^2,\varrho\rangle=\langle \varrho,\varrho\rangle^2=1=\ep(\varrho).
    \end{align*}
 
It remains to prove that the Hopf pairing is non-degenerate. Note that every  $x\in {\hat{\mathbf{U}}}_q^{\geq0}$ admits a unique decomposition $x=x_0+x_1\varrho$, with $x_0,x_1\in {\ol{U}}_q(\mathfrak{g})^{\ge0}$. Moreover, for any $b\in {\ol{U}}_q(\mathfrak{g})^{\ge0}$, we see that 
\begin{align*}
    \langle\varrho,b\rangle=\ep(b)=\langle b,\varrho\rangle.
\end{align*}
Now assume that $x\in {\hat{\mathbf{U}}}_q^{\geq0}$ satisfies $\langle x,y\rangle=0$ for all $y\in {\hat{\mathbf{U}}}_q^{\geq0}$. We show that $x=0$. 

Write $x=x_0+x_1\varrho$ with $x_0,x_1\in {\ol{U}}_q(\mathfrak{g})^{\ge0}$. For any $y\in {\ol{U}}_q(\mathfrak{g})^{\ge0}$, we obtain
\begin{align*}
    0=\langle x,y\rangle=\langle x_0,y\rangle+\langle x_1\varrho,y\rangle=\langle x_0,y\rangle+\langle x_1,y_{(1)}\rangle\ep(y_{(2)})=\langle x_0+x_1,y\rangle.
\end{align*}
Since $y\in {\ol{U}}_q(\mathfrak{g})^{\ge0}$ is arbitrary and the pairing on ${\ol{U}}_q(\mathfrak{g})^{\ge0}$ is non-degenerate (see \cite[Proposition 4.3]{LWY22}), we get 
\begin{align}\label{plus}
    x_0 +x_1=0.
\end{align}
For $y=y_0\varrho\in {\ol{U}}_q(\mathfrak{g})^{\ge0}\varrho$ with $y_0\in {\ol{U}}_q(\mathfrak{g})^{\ge0}$, we have 
\begin{align*}
    0&=\langle x,y_0\varrho\rangle\\&=\langle x_0,y_0\varrho\rangle+\langle x_1\varrho,y_0\varrho\rangle
    \\&=\langle x_0,y_0\rangle+\langle (x_1)_{(1)},(y_0)_{(1)}\rangle\langle (x_1)_{(2)},\varrho\rangle\langle \varrho,(y_0)_{(2)}\rangle\langle \varrho,\varrho\rangle
    \\&=\langle x_0,y_0\rangle-\langle x_1,y_0\rangle\\&=\langle x_0-x_1,y_0\rangle.
\end{align*}
As $y_0\in {\ol{U}}_q(\mathfrak{g})^{\ge0}$ is arbitrary, the non-degeneracy on ${\ol{U}}_q(\mathfrak{g})^{\ge0}$ implies 
\begin{align}\label{minus}
    x_0 -x_1=0.
\end{align}
Combining \eqref{plus} and \eqref{minus}, we obtain $x_0=x_1=0$, hence $x=0$. This proves that the left radical of the pairing is trivial; the right radical is treated analogously. Therefore the Hopf pairing is non-degenerate.
    \end{proof}

\section{\texorpdfstring{$\mathrm i$}{i}Hopf algebras associated with ${\hat{\mathbf{U}}}_q^{\geq0}$}\label{iHopf}

We now consider the iHopf algebra $({\hat{\mathbf{U}}}_q^{\geq0})^\imath_\tau$ associated with ${\hat{\mathbf{U}}}_q^{\geq0}$ and the Hopf pairing $\langle\tau(-),-\rangle$, i.e., for any $a,b \in ({\hat{\mathbf{U}}}_q^{\geq0})^\imath_\tau$, the multiplication is given by
\begin{align}\label{new product}
   a\circ b=\sum \langle \tau(b_{(2)}),a_{(1)}  \rangle a_{(2)} b_{(1)},  
 \end{align}
where $\Delta(a)=\sum a_{(1)}\otimes a_{(2)}$, $\Delta(b)=\sum b_{(1)}\otimes b_{(2)}$. Using the defining relations \eqref{basic relation} and \eqref{varrho'}, we obtain the following relations.


\begin{prp}\label{KE} 
    The following relations hold in $({\hat{\mathbf{U}}}_q^{\geq0})_\tau^\imath$:
    \begin{align*}
      &\genK_i\circ\genK_j
    =\genK_j\circ\genK_i,\quad
    \genK_i\circ\genK_i^{-1}=q_i^{-a_{i,\tau i}}
    =\genK_i^{-1}\circ\genK_i,\\    
    &\genK_i\circ\genE_j
    =q_i^{a_{i, j}-a_{i,\tau j}}\genE_j\circ\genK_i,\quad {\varrho\circ \varrho=-1,\quad x\circ \varrho=\varrho\circ x}, \mbox{ for } x\in ({\hat{\mathbf{U}}}_q^{\geq0})_\tau^\imath.
    \end{align*}
\end{prp}
\begin{proof}
    Applying  \eqref{hopfstructure}, \eqref{new product} and Proposition \ref{pairing}, a direct computation yields
    \begin{align*}
    \genK_i\circ\genK_j
    &=\langle\tau(\genK_j),\genK_i\rangle\genK_i\genK_j =q_i^{a_{i,\tau j}}\genK_i\genK_j=q_j^{a_{j,\tau i}}\genK_j\genK_i=\langle\tau(\genK_i),\genK_j\rangle\genK_j\genK_i=
    \genK_j\circ\genK_i,
      \\
      \genK_i\circ\genE_j
      &=\langle\varrho^{p(j)},\genK_i\rangle\genK_i\genE_j +\langle\tau(\genE_j),\genK_i\rangle\genK_j\genK_i
      =\genK_i\genE_j, \\
    \genE_j\circ\genK_i
    &=\langle\tau(\genK_i),\genE_j\rangle\varrho^{p(j)}\genK_i +\langle\tau(\genK_i),\genK_j\rangle\genE_j\genK_i
    =q_i^{a_{i,\tau j}}\genE_j\genK_i=q_i^{a_{i,\tau j}-a_{i,j}}\genK_i\genE_j,\\
    \varrho\circ\genE_i&=\langle\varrho^{p(i)},\varrho\rangle\varrho\cdot\genE_i +\langle\tau(\genE_i),\varrho\rangle\genK_i\varrho
      =(-1)^{p(i)}\varrho\cdot\genE_i,  \\    
      \genE_i\circ \varrho
      &=\langle\tau\varrho,\genE_i\rangle\varrho^{p(i)}\varrho +\langle\tau\varrho,\genK_i\rangle\genE_i\varrho
      =\genE_i\varrho=(-1)^{p(i)}\varrho\cdot\genE_i.
    \end{align*}
    The remaining relations can be verified similarly, and we omit the detail.
\end{proof}

To derive the other relations in 
$({\hat{\mathbf{U}}}_q^{\geq0})_\tau^\imath$, we rewrite the product of the generators $\genE_i$ ($i\in\mathbb{I}$)  under the multiplication in \eqref{new product}.

\begin{exa}\label{EE}
(1) For any $i,j\in \mathbb{I}$, we have
 \begin{align*}
        \genE_i\circ\genE_j
        &=\langle \varrho^{p(j)},\genE_i\rangle\varrho^{p(i)}\genE_j
        +\langle \tau(\genE_j),\genE_i\rangle\varrho^{p(i)}\genK_j
        +\langle \varrho^{p(j)},\genK_i\rangle\genE_i\genE_j
        +\langle \tau(\genE_j),\genK_i\rangle\genE_i\genK_j   \\    
        &=\genE_i\genE_j+\langle \tau(\genE_j),\genE_i\rangle\varrho^{p(i)}\genK_j.
    \end{align*}
   Since $\varrho^{p(i)}\circ\genK_j=\varrho^{p(i)}\genK_j$, we obtain 
   \begin{align*}
      \genE_i\genE_j= \genE_i\circ\genE_j-\langle \tau(\genE_j),\genE_i\rangle\varrho^{p(i)}\circ\genK_j.
   \end{align*}
(2) It is obvious that 
$$\genE_i\genE_j\circ\genE_k= \genE_i\circ\genE_j\circ\genE_k-\langle \tau(\genE_j),\genE_i\rangle\varrho^{p(i)}\circ\genK_j\circ\genE_k. $$
Using $$\Delta(\genE_i\genE_j)
    =\genE_i\genE_j\otimes\varrho^{p(i)}\varrho^{p(j)}+\genE_i\genK_j\otimes\varrho^{p(i)}\genE_j +\genK_i\genE_j\otimes\genE_i\varrho^{p(j)}+\genK_i\genK_j\otimes\genE_i\genE_j$$ and Proposition \ref{KE}, we have 
    \begin{align*}
     &\genE_i\genE_j\circ\genE_k
     \\
     =&\langle \varrho^{p(k)},\genK_i\genK_j\rangle\genE_i\genE_j\genE_k+\langle \tau(\genE_k),\genK_i\genE_j\rangle\genE_i\varrho^{p(j)}\genK_k+\langle \tau(\genE_k),\genE_i\genK_j\rangle\varrho^{p(i)}\genE_j\genK_k \\
     =&\genE_i\genE_j\genE_k
     +q_i^{a_{i,\tau k}}\langle \tau(\genE_k), \genE_j\rangle \genE_i\varrho^{p(j)}\genK_k
     +\langle \tau(\genE_k), \genE_i\rangle\varrho^{p(i)}\genE_j\genK_k\\
     =&\genE_i\genE_j\genE_k+ q_i^{a_{i,\tau k}-a_{i,k}}\langle \tau(\genE_k), \genE_j\rangle\varrho^{p(j)}\circ\genK_k\circ\genE_i+ {(-1)^{p(i)p(j)}}q_j^{-a_{j,k}}\langle \tau(\genE_k), \genE_i\rangle\varrho^{p(i)}\circ\genK_k\circ\genE_j.
    \end{align*}
    Hence,
    \begin{align*}
   \genE_i\genE_j\genE_k=\genE_i&\circ\genE_j\circ\genE_k   
   - {(-1)^{p(i)p(j)}}q_j^{-a_{j,k}}\langle \tau(\genE_k), \genE_i\rangle\varrho^{p(i)}\circ\genK_k\circ\genE_j\\&
   - q_i^{a_{i,\tau k}-a_{i,k}}\langle \tau(\genE_k), \genE_j\rangle\varrho^{p(j)}\circ\genK_k\circ\genE_i -\langle \tau(\genE_j),\genE_i\rangle\varrho^{p(i)}\circ\genK_j\circ\genE_k.
    \end{align*}

\end{exa}

Let $I=(i_1,\cdots,i_{l})\in {\mathbb{I}}^l$ be an ordered \(l\)-tuple of elements of \(I\). For notational simplicity,  we
 denote 
 $$E_I:=\genE_{i_1}\genE_{i_2}\cdots \genE_{i_l}=\prod\limits_{u=1}^l \genE_{i_u}, \quad E_I^{\circ}:=\genE_{i_1}\circ\genE_{i_2}\circ\cdots \circ\genE_{i_l}=\mathop{\circ\prod}\limits_{u=1}^l \genE_{i_u}.$$
 
\begin{prp}\label{new1}
{Suppose that $I=(i_1,\cdots,i_l)\in {\mathbb{I}}^l$. We have}
    \begin{align*}
   {E_I=E^{\circ}_{I}-\sum\limits_{v=2}^{l}\sum\limits_{u<v}d(u,v)\cdot(\prod\limits_{\substack{1\leq k<v\\k\neq u}}\genE_{i_k})\circ(\varrho^{p(i_u)}\circ\genK_{i_v})\circ(\mathop{\circ\prod}\limits_{k=v+1}^l\genE_{i_k}}),
     \end{align*}
   where $d(u,v)=\big(\prod\limits_{u<k<v}(-1)^{p(i_u)p(i_k)}\cdot q_{i_k}^{-a_{i_k,\tau i_{ v}}}\big)\cdot \langle \tau(\genE_{i_v}),\genE_{i_u}\rangle.$  
\end{prp}
\begin{proof}
Assume that $I{'}=(i_1,\cdots,i_{l-1})\in {\mathbb{I}}^{l-1}$. By \eqref{new product} and Proposition \ref{pairing}, a direct calculation shows that 
 \begin{align}
  &E_{I'}\circ \genE_{i_l}\nonumber\\
  =&\langle \varrho^{p(i_l)},\genK_{i_1}\genK_{i_2}\cdots\genK_{i_{l-1}}\rangle  \genE_{i_1}\cdots\genE_{i_{l-1}}\genE_{i_l}\nonumber\\&\qquad
  +\sum\limits_{{j}=1}^{l-1}\langle \tau(\genE_{i_l}), 
  (\prod\limits_{u=1}^{j-1}
  \genK_{i_u}) \cdot \genE_{i_{j}}\cdot
   (\prod\limits_{u=j+1}^{l-1}
 \genK_{i_u})\rangle 
(\prod\limits_{u=1}^{j-1}
  \genE_{i_u}) \cdot \varrho^{p(i_{j})}\cdot
   (\prod\limits_{u=j+1}^{l-1}
 \genE_{i_u})\cdot \genK_{i_l}\nonumber\\
  =&E_{I}+\sum\limits_{{j}=1}^{l-1}\big(\prod\limits_{u=1}^{{j}-1} \langle\tau(\genK_{i_l}), \genK_{i_{u}}\rangle \big)\langle\tau(\genE_{i_l}),\genE_{i_{j}} \rangle\cdot(-1)^{p(i_j)\sum\limits_{u=j+1}^{l-1}p(i_u)}\cdot\big(\prod\limits_{\substack{1\leq u<l\\u\neq j}}\genE_{i_u}\big)\cdot\varrho^{p(i_{j})}\genK_{i_l}\nonumber\\
  =&E_{I}+\sum\limits_{{j}=1}^{l-1}\big(\prod\limits_{u=1}^{{j}-1}
  q_{i_u}^{a_{{i_u},\tau{i_l}}}\big)\cdot
  \langle\tau(\genE_{i_l}),\genE_{i_{j}} \rangle
  \cdot(-1)^{p(i_j)\sum\limits_{u=j+1}^{l-1}p(i_u)}\cdot
  \big(\prod\limits_{\substack{1\leq u<l\\u\neq j}}\genE_{i_u}\big)\cdot\varrho^{p(i_{j})}\genK_{i_l}.\label{f1}
 \end{align}
 By Proposition \ref{KE}, we obtain
     \begin{align}
\big(\prod\limits_{\substack{1\leq u<l\\u\neq j}}\genE_{i_u}\big)\circ(\varrho^{p(i_{j})}\genK_{i_l})
&=\langle\tau(\varrho^{p(i_{j})}\genK_{i_l}),\prod\limits_{\substack{1\leq u<l\\u\neq j}}\genK_{i_u}\rangle\cdot\big(\prod\limits_{\substack{1\leq u<l\\u\neq j}}\genE_{i_u}\big)\cdot\varrho^{p(i_{j})}\genK_{i_l}\nonumber\\
&=(\prod\limits_{\substack{1\leq u<l\\u\neq j}}q_{i_u}^{a_{i_u,\tau i_l}})\cdot      
\big(\prod\limits_{\substack{1\leq u<l\\u\neq j}}\genE_{i_u}\big)\cdot\varrho^{p(i_{j})}\genK_{i_l}.\label{f2}
     \end{align}
    We now prove the proposition by induction on $l$. Example~\ref{EE} establishes the cases for $l=2,3$. For $l>3$, assume inductively that the proposition holds for all smaller values of $l$. The induction hypothesis yields
    \begin{align*}
        E_{I'}\circ \genE_{i_l}
        &=E^{\circ}_{I'}\circ E_{i_{l}}-\sum\limits_{v=2}^{l-1}\sum\limits_{u<v}d(u,v)\cdot(\prod\limits_{\substack{1\leq k<v\\k\neq u}}\genE_{i_k})\circ(\varrho^{p(i_u)}\circ\genK_{i_v})\circ(\mathop{\circ\prod}\limits_{k=v+1}^{l}\genE_{i_k})\\
        &=E^{\circ}_{I}-\sum\limits_{v=2}^{l-1}\sum\limits_{u<v}d(u,v)\cdot(\prod\limits_{\substack{1\leq k<v\\k\neq u}}\genE_{i_k})\circ(\varrho^{p(i_u)}\circ\genK_{i_v})\circ(\mathop{\circ\prod}\limits_{k=v+1}^{l}\genE_{i_k}).
    \end{align*}
     Then by \eqref{f1}-\eqref{f2} and $\varrho\circ\genK_i=\varrho\genK_i$, we have 
     \begin{align*}
      E_{I}
      &= E_{I'}\circ \genE_{i_l}- \sum\limits_{{j}=1}^{l-1}\big(\prod\limits_{u=1}^{{j}-1}
  q_{i_u}^{a_{{i_u},\tau{i_l}}}\big)\cdot
  \langle\tau(\genE_{i_l}),\genE_{i_{j}} \rangle
  \cdot(-1)^{p(i_j)\sum\limits_{u=j+1}^{l-1}p(i_u)}\cdot
  \big(\prod\limits_{\substack{1\leq u<l\\u\neq j}}\genE_{i_u}\big)\cdot\varrho^{p(i_{j})}\genK_{i_l}\\
  &=E_{I'}\circ \genE_{i_l}-\sum\limits_{{j}=1}^{l-1}
\big(\prod\limits_{u=j+1}^{l-1}(-1)^{p(i_j)p(i_u)}
  q_{i_u}^{-a_{{i_u},\tau{i_l}}}\big)\cdot  
  \langle\tau(\genE_{i_l}),\genE_{i_{j}} \rangle
  \cdot
  \big(\prod\limits_{\substack{1\leq u<l\\u\neq j}}\genE_{i_u}\big)\circ(\varrho^{p(i_{j})}\genK_{i_l})\\
  &=E^{\circ}_{I}-\sum\limits_{v=2}^{l-1}\sum\limits_{u<v}d(u,v)\cdot(\prod\limits_{\substack{1\leq k<v\\k\neq u}}\genE_{i_k})\circ(\varrho^{p(i_u)}\circ\genK_{i_v})\circ(\mathop{\circ\prod}\limits_{k=v+1}^{l}\genE_{i_k})\\&\quad\qquad-\sum\limits_{{j}=1}^{l-1}
  d(j,l)
  \big(\prod\limits_{\substack{1\leq u<l\\u\neq j}}\genE_{i_u}\big)\circ(\varrho^{p(i_{j})}\circ\genK_{i_l})\\
  &=E^{\circ}_{I}-\sum\limits_{v=2}^{l}\sum\limits_{u<v}d(u,v)\cdot(\prod\limits_{\substack{1\leq k<v\\k\neq u}}\genE_{i_k})\circ(\varrho^{p(i_u)}\circ\genK_{i_v})\circ(\mathop{\circ\prod}\limits_{k=v+1}^{l}\genE_{i_k})
     \end{align*}
     as desired.
\end{proof}

 Set $S=\{1,2,\cdots,l\}$. For any $1\leq s\leq \lfloor \frac{l}{2} \rfloor$, let $M$ be a set consisting of $s$ ordered pairs
 \begin{align}\label{def:m}
     M=\{(x_j,y_j)\ |\ x_j,y_j\in S,x_j<y_j, j=1,\cdots s\},
 \end{align}
satisfying:
\begin{itemize}
     \item[-] Distinct vertices: the $2s$  elements $x_1,y_1,x_2,y_2,\cdots,x_s,y_s$
   are pairwise distinct with $y_1<y_2<\cdots<y_s$;
    \item[-] Order preservation: $(x_u,y_u)>(x_v,y_v)$ whenever $y_u>y_v$.
\end{itemize}
Denote by $\mathcal{M}_{s}(S)$ the set of all such $M$. Put $T(M)=\{x_1,y_1,x_2,y_2,\cdots,x_s,y_s\}$. 
  For any $M\in\mathcal{M}_{s}(S)$, let 
 \begin{align*}
    E_{I^M}^{\circ}:= \mathop{\circ\prod}\limits_{\substack{1\leq u\leq l\\u\notin T(M)}}\genE_{i_u}.
 \end{align*}
{We now introduce the following notation:  for any $(x,y)\in M$, define the coefficient functions by}
 \begin{align}\label{def1}
    &f({x,y}) =p(i_{x})\Big(\sum\limits_{\substack{x<u<y\\u\in S}}p(i_{u})-\sum\limits_{{\substack{x<x'<y<y'\\(x',y')\in M}}}p(i_{x'})\Big),\\
\label{def2}    &g({x,y})=-\sum\limits_{\substack{x<u<y\\u\in S}}a_{i_{y}, i_u} +\sum\limits_{{\substack{x<x'<y<y'\\(x',y')\in M}}}a_{i_{y}, i_{x'}},\\
    \label{def3}&\varphi(x,y)= (-1)^{f({x,y})}\cdot 
    q_{i_{y}}^{g({x,y})}\cdot(\prod\limits_{\substack{u<x\\u\neq x',(x',y')>(x,y)}}q_{i_u}^{a_{i_u,\tau i_y}-a_{i_u, i_y}})\cdot\langle\tau(\genE_{i_y}),\genE_{i_x}\rangle.
\end{align}

 \begin{thm}\label{formula}
     Suppose that $I=(i_1,\cdots,i_l)\in {\mathbb{I}}^l$. Then
     $$E_{I}=E^{\circ}_{I}+o(E^{\circ}_{I}),$$
     where $o(E^{\circ}_{I})$ is given by
     {{\begin{align*}
    \sum\limits_{s=1}^{\lfloor \frac{l}{2} \rfloor}\sum\limits_{M\in \mathcal{M}_s(S)} (-1)^s\Big(\mathop{\circ\prod}\limits_{(x, y)\in M}\varphi(x,y)\cdot \varrho^{p(i_{x})}\circ \genK_{i_{y}}\Big)\circ E_{I^M}^{\circ}.
     \end{align*}}}
 \end{thm}

\begin{proof}
 Assume that $I{'}=(i_1,\cdots,i_{l-1})\in {\mathbb{I}}^{l-1}$. 
 We now prove the theorem by induction on $l$. 
Indeed, Example \ref{EE} verifies the cases $l=2, 3$. Now assume $l>3$; then by induction, we have $$E_{I'}\circ\genE_{i_l}=E^{\circ}_{I'}\circ\genE_{i_l}+o(E^{\circ}_{I'})\circ\genE_{i_l},$$
 where $$o(E^{\circ}_{I'})=
       \sum\limits_{s=1}^{\lfloor \frac{l-1}{2} \rfloor}\sum\limits_{{M'}\in \mathcal{M}_s(S\setminus\{l\})}(-1)^s\Big(\mathop{\circ\prod}\limits_{(x, y)\in {M'}}\varphi'(x,y) 
\varrho^{p(i_{x})}\circ \genK_{i_{y}}\Big)\circ E_{I'^{M'}}^{\circ}.$$

By \eqref{f1}-\eqref{f2}, we have
 \begin{align*}   
 E_{I} 
 &=E_{I'}\circ \genE_{i_l}-\sum\limits_{{j}=1}^{l-1} d(j,l)(
  \prod\limits_{\substack{1\leq u<l\\u\neq j}}\genE_{i_u})\circ(\varrho^{p(i_{j})}\circ\genK_{i_l})\\
&=E^{\circ}_{I'}\circ\genE_{i_l}+o(E^{\circ}_{I'})\circ\genE_{i_l}
 -\sum\limits_{{j}=1}^{l-1} d(j,l) (\prod\limits_{\substack{1\leq u<l\\u\neq j}}q_{i_u}^{a_{i_u,\tau i_l}-a_{i_u,i_l}}) (\varrho^{p(i_{j})}\circ\genK_{i_l})\circ(
  \prod\limits_{\substack{1\leq u<l\\u\neq j}}\genE_{i_u})\\
  &=E^{\circ}_{I}+o(E^{\circ}_{I'})\circ\genE_{i_l}
 -\sum\limits_{{j}=1}^{l-1} \varphi(j,l) (\varrho^{p(i_{j})}\circ\genK_{i_l})\circ(
  \prod\limits_{\substack{1\leq u<l\\u\neq j}}\genE_{i_u}),
 \end{align*} 
 where 
 $$d(j,l)=\big(\prod\limits_{j<k<l}(-1)^{p(i_j)p(i_k)}\cdot q_{i_k}^{-a_{i_k,\tau i_{ l}}}\big)\cdot \langle \tau(\genE_{i_l}),\genE_{i_j}\rangle.$$ 
Then by induction, we obtain 
  {\small\begin{align*}
(\varrho^{p(i_{j})}\circ\genK_{i_l})\circ(\prod\limits_{\substack{1\leq u<l\\u\neq j}}\genE_{i_u})
  =(\varrho^{p(i_{j})}\circ\genK_{i_l})\circ(
  \mathop{\circ\prod}\limits_{\substack{1\leq u<l\\u\neq j}}\genE_{i_u})+(\varrho^{p(i_{j})}\circ\genK_{i_l})\circ o(
  \mathop{\circ\prod}\limits_{\substack{1\leq u<l\\u\neq j}}\genE_{i_u}),
  \end{align*}}
where 
\begin{align*}
 o(
  \mathop{\circ\prod}\limits_{\substack{1\leq u<l\\u\neq j}}\genE_{i_u})
  &=\sum\limits_{s=1}^{\lfloor \frac{l-2}{2} \rfloor}\sum\limits_{{M''}\in \mathcal{M}_s(S \setminus \{j,l\})}(-1)^s\Big(\mathop{\circ\prod}\limits_{(x, y)\in {M''}} \varphi''(x,y) \varrho^{p(i_{x})}\circ \genK_{i_{y}}\Big)\circ E_{I'^{M''}}^\circ.   
\end{align*}

  It follows from the coefficient functions that $$\varphi(j,l)\cdot \prod\limits_{(x, y)\in {M''}} \varphi''(x,y)
  =\prod\limits_{(x, y)\in {M}}\varphi(x,y)$$ 
  for $M''\in \mathcal{M}_s(S \setminus \{j,l\})$, $M=M''\cup \{(j,l)\}$.  Using Proposition \ref{KE}, we have
  { \begin{align*}
o(\varrho KE):
&=-\sum\limits_{{j}=1}^{l-1} \varphi(j,l) (\varrho^{p(i_{j})}\circ\genK_{i_l})\circ(
  \prod\limits_{\substack{1\leq u<l\\u\neq j}}\genE_{i_u})\\
&=-\sum\limits_{{j}=1}^{l-1} \varphi(j,l)(\varrho^{p(i_{j})}\circ\genK_{i_l})\circ(
  \mathop{\circ\prod}\limits_{\substack{1\leq u<l\\u\neq j}}\genE_{i_u})-\sum\limits_{{j}=1}^{l-1} \varphi(j,l)\cdot (\varrho^{p(i_{j})}\circ\genK_{i_l})\circ o(
  \mathop{\circ\prod}\limits_{\substack{1\leq u<l\\u\neq j}}\genE_{i_u})\\
  &= \sum\limits_{s=1}^{\lfloor \frac{l}{2} \rfloor}\sum\limits_{\substack{M\in \mathcal{M}_s(S)\\y_s=l}}(-1)^s\Big(\mathop{\circ\prod}\limits_{{(x, y)\in M}} \varphi(x,y)\cdot 
    \varrho^{p(i_{x})}\circ \genK_{i_{y}}\Big) \circ E_{I^M}^\circ.
   \end{align*}}
   Note that $\varphi(x,y)=\varphi'(x,y)$ for $y\neq l$. We have
   \begin{align*}
    o(E^{\circ}_{I'})\circ\genE_{i_l}&= \sum\limits_{s=1}^{\lfloor \frac{l}{2} \rfloor}\sum\limits_{{{M'}\in \mathcal{M}_s(S\setminus\{l\})}}(-1)^s\Big(\mathop{\circ\prod}\limits_{\substack{(x, y)\in {M'}}} \varphi'(x,y)\cdot 
    \varrho^{p(i_{x})}\circ \genK_{i_{y}}\Big)\circ  E_{I^{M'}}^\circ\\
    &=\sum\limits_{s=1}^{\lfloor \frac{l}{2} \rfloor}\sum\limits_{\substack{{M}\in \mathcal{M}_s(S)\\y_s\neq l}}(-1)^s\Big(\mathop{\circ\prod}\limits_{(x, y)\in M} \varphi(x,y)\cdot 
    \varrho^{p(i_{x})}\circ \genK_{i_{y}}\Big)\circ  E_{I^M}^\circ.  
   \end{align*}
   Thus we obtain
   \begin{align*}
     o(E^{\circ}_{I'})\circ\genE_{i_l}+  o(\varrho KE)=o(E^{\circ}_{I})
   \end{align*}
   as desired.
\end{proof}
\begin{rmk}
  The lower-order terms appearing in the higher order quantum Serre relations can be computed explicitly using Theorem~\ref{formula}; see the last section.
\end{rmk}

\section{Realization of \texorpdfstring{$\mathrm i$}{i}quantum supergroups via \texorpdfstring{$\mathrm i$}{i}Hopf algebras}\label{realization}
In this section, we show that the quasi-split iquantum supergroup of basic type can be realized as the iHopf algebra associated with the Borel quantum supergroup. We further provide the iHopf realization of the quantum supersymmetric pair. Following \cite[Remark 2.4]{SW25}, throughout the remainder of the paper, we assume that $\tau i\neq i$ whenever $i\in \mathbb{I}_\iso$ for any given involution $\tau$.

Recall the iHopf algebra of diagonal type from Section \ref{s2-4}.
The iHopf algebra $({\hat{\mathbf{U}}}_q^{\geq0}\otimes{\hat{\mathbf{U}}}_q^{\geq0})^{\imath}$ is a Hopf algebra with the following operations
\begin{align}\label{multiplication}\begin{split}
    (a\otimes b)\circ(c\otimes d)&=\sum\langle a_{(1)},d_{(2)}\rangle\cdot\langle c_{(2)},b_{(1)}\rangle\cdot a_{(2)}c_{(1)}\otimes b_{(2)}d_{(1)},\\
    \Delta^{\imath}(a\otimes b)&=\langle a_{(2)},b_{(2)}\rangle\cdot(a_{(1)}\otimes b_{(3)})\otimes(a_{(3)}\otimes b_{(1)}),\\
    \ep^{\imath}(a\otimes b)&=\langle a,S^{-1}(b)\rangle,\\
    S^{\imath}(a\otimes b)&= S(a)\otimes S^{-1}(b),
\end{split}
\end{align}
for any $a,b,c,d\in {\hat{\mathbf{U}}}_q^{\geq0}$. 

The iHopf algebra of diagonal type yields the following realization of the quantum supergroup ${\mathbf{U}}$.

\begin{thm}\label{Phi}
    {There is a Hopf algebra {isomorphism}} $\Phi: {{\mathbf{U}}}\to ({\hat{\mathbf{U}}}_q^{\geq0}\otimes{\hat{\mathbf{U}}}_q^{\geq0})^{\imath}$ given by 
    \begin{align*}
        &\genE_i\mapsto \varrho^{p(i)}\genE_i\otimes\varrho^{p(i)},\quad 
        \genF_i\mapsto 1\otimes \genE_i\varrho^{p(i)},\quad 
        \genK_i\mapsto (-1)^{p(i)}(\varrho^{p(i)}\genK_i\otimes\varrho^{p(i)}),\\
        &\genK'_i\mapsto 1\otimes\genK_i,\qquad
        \varrho\mapsto1\otimes \varrho,\qquad
         \varrho'\mapsto\varrho\otimes 1.
    \end{align*}
\end{thm}
\begin{proof}
    We first verify that $\Phi$ is an algebra homomorphism. Applying \eqref{multiplication}, we have
    \begin{align*}
     \Phi(\genK_i)\circ\Phi(\genK_j)
    &=[(-1)^{p(i)}(\varrho^{p(i)}\genK_i\otimes\varrho^{p(i)})] \circ[(-1)^{p(j)}(\varrho^{p(j)}\genK_j\otimes\varrho^{p(j)})]\\
    &=(-1)^{p(i)p(j)}\langle \varrho^{p(i)}\genK_i,\varrho^{p(j)}\rangle \langle\varrho^{p(j)}\genK_j,\varrho^{p(i)}\rangle \varrho^{p(i)}\genK_i \varrho^{p(j)}\genK_j\otimes\varrho^{p(i)}\varrho^{p(j)}\\
    &=[(-1)^{p(j)}(\varrho^{p(j)}\genK_j\otimes\varrho^{p(j)})]\circ[(-1)^{p(i)}(\varrho^{p(i)}\genK_i\otimes\varrho^{p(i)}) ]
    \\&
    = \Phi(\genK_j)\circ\Phi(\genK_i),\\
   \Phi(\genK_i)\circ\Phi(\genK'_j)
   &=[(-1)^{p(i)}(\varrho^{p(i)}\genK_i\otimes\varrho^{p(i)})] \circ(1\otimes\genK_j)\\&
   =(-1)^{p(i)}\langle\varrho^{p(i)}\genK_i,\genK_j\rangle \langle1,\varrho^{p(i)}\rangle\varrho^{p(i)}\genK_i\otimes\varrho^{p(i)}\genK_j\\
   &= (1\otimes\genK_j)\circ[(-1)^{p(i)}(\varrho^{p(i)}\genK_i\otimes\varrho^{p(i)})]\\&
   =\Phi(\genK'_j)\circ\Phi(\genK_i),\\
   \Phi(\genK_i)\circ\Phi(\genE_j)       
   &=[(-1)^{p(i)}(\varrho^{p(i)}\genK_i\otimes\varrho^{p(i)})]\circ(\varrho^{p(j)}\genE_j\otimes\varrho^{p(j)})\\&
   =(-1)^{p(i)}\langle\varrho^{p(i)}\genK_i,\varrho^{p(j)}\rangle \langle1,\varrho^{p(i)}\rangle \varrho^{p(i)}\genK_i\varrho^{p(j)}\genE_j\otimes\varrho^{p(i)}\varrho^{p(j)}\\
   &=(-1)^{p(i)}q_i^{a_{i,j}}\langle\varrho^{p(j)}\genK_j,\varrho^{p(i)}\rangle \langle\varrho^{p(i)}\genK_i,\varrho^{p(j)}\rangle\varrho^{p(j)}\genE_j\varrho^{p(i)}\genK_i \otimes\varrho^{p(i)}\varrho^{p(j)}\\&
   =q_i^{a_{i,j}}(\varrho^{p(j)}\genE_j\otimes\varrho^{p(j)})\circ[(-1)^{p(i)}(\varrho^{p(i)}\genK_i\otimes\varrho^{p(i)})]\\
   &=q_i^{a_{i,j}}\Phi(\genE_j)\circ\Phi(\genK_i),\\
   \Phi(\genK_i)\circ\Phi(\genF_j)
   &=[(-1)^{p(i)}(\varrho^{p(i)}\genK_i\otimes\varrho^{p(i)})]\circ(1\otimes\genE_j\varrho^{p(j)})\\&
   =(-1)^{p(i)}\langle\varrho^{p(i)}\genK_i,1\rangle \langle1,\varrho^{p(i)}\rangle\varrho^{p(i)}\genK_i\otimes\varrho^{p(i)}\genE_j\varrho^{p(j)}\\
   &=(-1)^{p(i)}q_i^{-a_{i,j}}\langle1,\varrho^{p(i)}\rangle\langle\varrho^{p(i)}\genK_i,\varrho^{p(j)}\genK_j\rangle\varrho^{p(i)}\genK_i\otimes\varrho^{p(j)}\genE_j\varrho^{p(i)}\\
   &=(1\otimes\genE_j\varrho^{p(j)})\circ[(-1)^{p(i)}(\varrho^{p(i)}\genK_i\otimes\varrho^{p(i)})]
   \\
   &=q_i^{-a_{i,j}}\Phi(\genF_j)\circ\Phi(\genK_i),
   \\
   \Phi(\genK'_i)\circ\Phi(\genE_j)
   &=(1\otimes\genK_i)\circ(\varrho^{p(j)}\genE_j\otimes\varrho^{p(j)})\\
   &=\langle1,\varrho^{p(j)}\rangle \langle1,\genK_i\rangle\varrho^{p(j)}\genE_j\otimes\genK_i\varrho^{p(j)}\\
   &=q_i^{-a_{i,j}}\langle\varrho^{p(j)}\genK_j,\genK_i\rangle \langle1,\varrho^{p(j)}\rangle\varrho^{p(j)}\genE_j\otimes\varrho^{p(j)}\genK_i\\
   &=q_i^{-a_{i,j}}(\varrho^{p(j)}\genE_j\otimes\varrho^{p(j)})\circ(1\otimes\genK_i)
   \\
   &=q_i^{-a_{i,j}}\Phi(\genE_j)\circ\Phi(\genK'_i),\\
   \Phi(\varrho)\circ\Phi(\varrho)
   &=(1\otimes\varrho)\circ(1\otimes\varrho)
   =\langle1,\varrho\rangle^2 1\otimes\varrho^2
   =1\otimes1,\\
   \Phi(\genE_i)\circ\Phi(\varrho)
   &=(\varrho^{p(i)}\genE_i\otimes\varrho^{p(i)})\circ(1\otimes\varrho)
   =\langle\varrho^{p(i)}\genK_i,\varrho\rangle \langle1,\varrho^{p(i)}\rangle\varrho^{p(i)}\genE_i\otimes\varrho^{p(i)}\varrho\\
   &=(-1)^{p(i)}\langle1,\varrho^{p(i)}\rangle \langle1,\varrho\rangle\varrho^{p(i)}\genE_i\otimes\varrho\varrho^{p(i)}
   =(1\otimes\varrho)\circ(\varrho^{p(i)}\varrho(\genE_i)\otimes\varrho^{p(i)})\\&
   =\Phi(\varrho)\circ\Phi(\varrho(\genE_i)),\\
   \Phi(\genK_i)\circ\Phi(\varrho)
   &=[(-1)^{p(i)}(\varrho^{p(i)}\genK_i\otimes\varrho^{p(i)})]\circ(1\otimes\varrho)\\
   &=(-1)^{p(i)}\langle\varrho^{p(i)}\genK_i,\varrho\rangle \langle1,\varrho^{p(i)}\rangle \varrho^{p(i)}\genK_i\otimes\varrho^{p(i)}\varrho\\
   &=(-1)^{p(i)}\langle1,\varrho^{p(i)}\rangle \langle\varrho^{p(i)}\genK_i,\varrho\rangle \varrho^{p(i)}\genK_i\otimes\varrho^{p(i)}\varrho\\
   &=(1\otimes\varrho)\circ[(-1)^{p(i)}(\varrho^{p(i)}\genK_i\otimes\varrho^{p(i)})]\\&
   =\Phi(\varrho)\circ\Phi(\genK_i),\\
   \Phi(\genE_i)\circ\Phi(\varrho')
   &=(\varrho^{p(i)}\genE_i\otimes\varrho^{p(i)})\circ(\varrho\otimes1)
   =\langle\varrho^{p(i)}\genK_i,1\rangle \langle\varrho,\varrho^{p(i)}\rangle\varrho^{p(i)}\genE_i\varrho\otimes\varrho^{p(i)}\\
   &=(-1)^{p(i)}\langle\varrho,\varrho^{p(i)}\rangle \langle1,1\rangle\varrho\varrho^{p(i)}\genE_i\otimes\varrho^{p(i)}
   =(\varrho\otimes1)\circ(\varrho^{p(i)}\varrho(\genE_i)\otimes\varrho^{p(i)})\\&
   =\Phi(\varrho')\circ\Phi(\varrho(\genE_i)),\\
   \Phi(\genK'_i)\circ\Phi(\varrho')
   &=(1\otimes\genK_i)\circ(\varrho\otimes1)=\varrho\otimes\genK_i=\Phi(\varrho')\circ\Phi(\genK'_i).
    \end{align*}
 We also have
\begin{align*}
\Phi(\genE_i)&\circ\Phi(\genF_j)-(-1)^{p(i)p(j)}\Phi(\genF_j)\circ\Phi(\genE_i)\\
   &=(\varrho^{p(i)}\genE_i\otimes\varrho^{p(i)})\circ(1\otimes\genE_j\varrho^{p(j)})-(-1)^{p(i)p(j)}(1\otimes\genE_j\varrho^{p(j)})\circ(\varrho^{p(i)}\genE_i\otimes\varrho^{p(i)})\\
   &=\langle\varrho^{p(i)}\genE_i,\genE_j\varrho^{p(j)}\rangle\langle\varrho^{p(i)},1\rangle1\otimes\varrho^{p(i)}\varrho^{p(j)}\genK_j\\&\qquad
   +\langle\varrho^{p(i)}\genK_i,1\rangle\langle1,\varrho^{p(i)}\rangle\varrho^{p(i)}\genE_i\otimes\varrho^{p(i)}\genE_j\varrho^{p(j)}\\&\qquad
   -(-1)^{p(i)p(j)}[\langle1,\varrho^{p(i)}\rangle\langle1,\genK_j\varrho^{p(j)}\rangle\varrho^{p(i)}\genE_i\otimes\genE_j\varrho^{p(j)}\varrho^{p(i)}\\&\qquad
   +\langle1,\varrho^{p(i)}\rangle\langle\varrho^{p(i)}\genE_i,\genE_j\varrho^{p(j)}\rangle \varrho^{p(i)}\genK_i\otimes\varrho^{p(i)}
   ]\\
   &=-\delta_{ij}\frac{1\otimes\varrho^{p(i)}\varrho^{p(j)}\genK_j}{q_i-q_i^{-1}}
   +(-1)^{p(i)p(j)}\delta_{ij}\frac{\varrho^{p(i)}\genK_i\otimes\varrho^{p(i)}}{q_i-q_i^{-1}} 
   \\
   &=\delta_{ij}\frac{(-1)^{p(i)}(\varrho^{p(i)}\genK_i\otimes\varrho^{p(i)})-1\otimes\genK_i}{q_i-q_i^{-1}}
   \\
&=\delta_{ij}\frac{\Phi(\genK_i)-\Phi(\genK'_i)}{q_i-q_i^{-1}}.
\end{align*}
   Similarly, the remaining relations in \eqref{basic relation} and \eqref{varrho'} 
   \begin{align*}
    &  \Phi(\genK'_i)\circ\Phi(\genK'_j)=\Phi(\genK'_j)\circ\Phi(\genK'_i),\qquad \Phi(\genK'_i)\circ\Phi(\genF_j)=q_i^{a_{i,j}}\Phi(\genF_j)\circ\Phi(\genK'_i),\\
      &\Phi(\genF_i)\circ\Phi(\varrho)
   =\Phi(\varrho)\circ\Phi(\varrho(\genF_i)),\qquad
   \Phi(\genK'_i)\circ\Phi(\varrho)
   =\Phi(\varrho)\circ\Phi(\genK'_i),\\
   &\Phi(\genF_i)\circ\Phi(\varrho')
   =\Phi(\varrho')\circ\Phi(\varrho(\genF_i)),\qquad
   \Phi(\genK_i)\circ\Phi(\varrho')
   =\Phi(\varrho')\circ\Phi(\genK_i),
   \\
   &\Phi(\varrho)\circ\Phi(\varrho')
   =\Phi(\varrho')\circ\Phi(\varrho),\qquad\qquad
   \Phi(\varrho')\circ\Phi(\varrho')
   =1\otimes 1,
   \end{align*}
   can be checked case by case.

    As for the Serre relations in Proposition \ref{P:SerreRelations}, we claim that 
    \begin{align*}
        \Phi(\genE_{i_1})\circ \Phi(\genE_{i_2})\circ\cdots\circ\Phi(\genE_{i_l})
        =\varrho^{\sum\limits_{u=1}^l p(i_u)}(\prod\limits_{u=1}^l\genE_{i_u} )\otimes\varrho^{\sum\limits_{u=1}^l p(i_u)},
    \end{align*}
 where $i_1,\cdots,i_l\in\mathbb{I}$. We prove it by induction on $l$. Indeed, if $l=2$, then we have 
 \begin{align*}
   \Phi(\genE_{i_1})\circ \Phi(\genE_{i_2})
   &=(\varrho^{p({i_1})}\genE_{i_1}\otimes\varrho^{p({i_1})})\circ(\varrho^{p({i_2})}\genE_{i_2}\otimes\varrho^{p({i_2})})\\
   &=\langle\varrho^{p({i_1})}\genK_{i_1},\varrho^{p({i_2})}\rangle\langle1,\varrho^{p({i_1})}\rangle\varrho^{p({i_1})}\genE_{i_1}\varrho^{p({i_2})}\genE_{i_2}\otimes\varrho^{p({i_1})}\varrho^{p({i_2})}\\
   &=\varrho^{p({i_1})}\varrho^{p({i_2})}\genE_{i_1}\genE_{i_2}\otimes\varrho^{p({i_1})}\varrho^{p({i_2})}.
 \end{align*}
 Now assume $l>2$, then by induction we obtain
 \begin{align*}
    \Phi(\genE_{i_1})&\circ \Phi(\genE_{i_2})\circ\cdots\circ\Phi(\genE_{i_l})\circ\Phi(\genE_{i_{l+1}})\\
        &=[\varrho^{\sum\limits_{u=1}^l p(i_u)}(\prod\limits_{u=1}^l\genE_{i_u} )\otimes\varrho^{\sum\limits_{u=1}^l p(i_u)}] \circ(\varrho^{p({i_{l+1}})}\genE_{i_{l+1}}\otimes\varrho^{p({i_{l+1}})})\\
        &=\langle\varrho^{\sum\limits_{u=1}^l p(i_u)}\prod\limits_{u=1}^l\genK_{i_u} ,\varrho^{p({i_{l+1}})}\rangle\langle1,\varrho^{\sum\limits_{u=1}^l p(i_u)}\rangle\varrho^{\sum\limits_{u=1}^l p(i_u)}(\prod\limits_{u=1}^{l}\genE_{i_u})\varrho^{p({i_{l+1}})}\genE_{i_{l+1}}\otimes\varrho^{\sum\limits_{u=1}^{l+1} p(i_u)}\\
        &=\varrho^{\sum\limits_{u=1}^{l+1} p(i_u)}(\prod\limits_{u=1}^{l+1}\genE_{i_u})\otimes\varrho^{\sum\limits_{u=1}^{l+1} p(i_u)}.
 \end{align*}

 Note that $\Phi$ is also a coalgebra homomorphism. Indeed, we have
 \begin{align*}
     \Delta^{\imath}\Phi(\genE_i)
     &=\Delta^{\imath}(\varrho^{p(i)}\genE_i\otimes\varrho^{p(i)})\\
     &=\langle1,\varrho^{p(i)}\rangle(\varrho^{p(i)}\genE_i\otimes\varrho^{p(i)})\otimes(1\otimes\varrho^{p(i)})\\
     &\qquad+\langle\varrho^{p(i)}\genK_i,\varrho^{p(i)}\rangle(\varrho^{p(i)}\genK_i\otimes\varrho^{p(i)})\otimes(\varrho^{p(i)}\genE_i\otimes\varrho^{p(i)})\\
     &=(\varrho^{p(i)}\genE_i\otimes\varrho^{p(i)})\otimes(1\otimes\varrho^{p(i)})
     +(-1)^{p(i)}(\varrho^{p(i)}\genK_i\otimes\varrho^{p(i)})\otimes(\varrho^{p(i)}\genE_i\otimes\varrho^{p(i)})\\
     &=(\Phi\otimes\Phi)(\genE_i\otimes\varrho^{p(i)}+\genK_i\otimes\genE_i)\\
     &=(\Phi\otimes\Phi)\Delta(\genE_i),
     \\
     \Delta^{\imath}\Phi(\genF_i)
     &=\Delta^{\imath}(1\otimes \genE_i\varrho^{p(i)})\\
     &=\langle1,1\rangle(1\otimes1)\otimes(1\otimes \genE_i\varrho^{p(i)})
     +\langle1,\genK_i\varrho^{p(i)}\rangle(1\otimes \genE_i\varrho^{p(i)})\otimes(1\otimes\genK_i\varrho^{p(i)})\\
     &=(1\otimes1)\otimes(1\otimes \genE_i\varrho^{p(i)})
     +(1\otimes \genE_i\varrho^{p(i)})\otimes(1\otimes\genK_i\varrho^{p(i)})\\
     &=(\Phi\otimes\Phi)(\genF_i\otimes\varrho^{p(i)}\genK'_i+1\otimes\genF_i)
     \\
     &=(\Phi\otimes\Phi)\Delta(\genF_i),
     \\
     \Delta^{\imath}\Phi(\genK_i)
     &=\Delta^{\imath}((-1)^{p(i)}(\varrho^{p(i)}\genK_i\otimes\varrho^{p(i)}))\\
     &=(-1)^{p(i)}\langle\varrho^{p(i)}\genK_i,\varrho^{p(i)}\rangle(\varrho^{p(i)}\genK_i\otimes\varrho^{p(i)})\otimes(\varrho^{p(i)}\genK_i\otimes\varrho^{p(i)})\\
     &=[(-1)^{p(i)}(\varrho^{p(i)}\genK_i\otimes\varrho^{p(i)})]\otimes[(-1)^{p(i)}(\varrho^{p(i)}\genK_i\otimes\varrho^{p(i)})]\\
     &=(\Phi\otimes\Phi)(\genK_i\otimes\genK_i)\\
     &=(\Phi\otimes\Phi)\Delta(\genK_i),\\
      \Delta^{\imath}\Phi(\genK'_i)
     &=\Delta^{\imath}(1\otimes \genK_i)
     =\langle1,\genK_i\rangle(1\otimes \genK_i)\otimes(1\otimes \genK_i)
     =(\Phi\otimes\Phi)\Delta(\genK'_i),\\     
     \Delta^{\imath}\Phi(\varrho)
     &=\Delta^{\imath}(1\otimes \varrho)
     =\langle1,\varrho\rangle(1\otimes \varrho)\otimes(1\otimes \varrho)
     =(\Phi\otimes\Phi)\Delta(\varrho),\\
     \Delta^{\imath}\Phi(\varrho')
     &=\Delta^{\imath}(\varrho\otimes 1)
     =\langle\varrho,1\rangle(\varrho\otimes1 )\otimes(\varrho\otimes1 )
     =(\Phi\otimes\Phi)\Delta(\varrho').
 \end{align*}
 
 It can be directly verified that the restrictions ${\Phi}|_{\mathbf{U}_q^{-}}$, ${\Phi}|_{\mathbf{U}_q^{0}}$ and ${\Phi}|_{\mathbf{U}_q^{+}}$ are injective. Under the triangular decomposition $\mathbf{U}\cong \mathbf{U}_q^{-}\otimes \mathbf{U}_q^{0}\otimes\mathbf{U}_q^{+}$, we have 
 \begin{align*}
     \mbox{Ker}(\Phi)\cong \mbox{Ker}({\Phi}|_{\mathbf{U}_q^{-}})\otimes \mbox{Ker}({\Phi}|_{\mathbf{U}_q^{0}})\otimes\mbox{Ker}({\Phi}|_{\mathbf{U}_q^{+}})=0.
 \end{align*}
Thus $\Phi$ is injective. Since ${\Phi}$ is also surjective, we deduce that $\Phi$ is an isomorphism of Hopf algebras. 
\end{proof}

For a Hopf algebra $H$ with a Hopf pairing $\psi(-,-)$ and a Hopf algebra endomorphism $\tau$ of $H$, recall from \cite[Definition 2.10]{CLPRW25} that a linear map $\chi:H\rightarrow \mathbb{F}$ is called $\tau$-twisted compatible if  $$\chi(ab)=\sum\chi(a_{(1)})\chi(b_{(2)})\psi( \tau(a_{(2)}),b_{(1)})$$ 
for all $a,b\in H$. $\chi$ is a compatible map if it is $\tau$-twisted compatible with $\tau=$Id.

We show that there is a $\tau$-twisted compatible map on ${\hat{\mathbf{U}}}_q^{\geq0}$ and an involution $\tau$.

\begin{prp}\label{chi}
  { Let $\tau$ be an involution in Inv($A$) with $\tau i\neq i$ if $i\in\mathbb{I}_\iso$. Then there is a unique $\tau$-twisted compatible map $\chi:{\hat{\mathbf{U}}}_q^{\geq0}\rightarrow \C(q)$ such that }
    \begin{align*}
        & \chi(1)=1,\quad \chi(\genE_i)=0, \quad \chi(\genK_i)=
          \langle \genK_i,  \genK_{\tau i}\rangle,\quad \chi(\varrho)=\sqrt{-1}, \quad \mbox{ for all }~i\in\mathbb{I}.
    \end{align*}
\end{prp}
\begin{proof}
We shall prove that $\chi$ vanishes on the ideal generated by the relations  involving $E_i,K_i,\varrho$ in \eqref{basic relation} and \eqref{varrho'}, together with  the higher order quantum Serre relations in Proposition \ref{P:SerreRelations}. 
    By Lemma \ref{Delta}, Proposition \ref{pairing}, and \cite[Lemma 4.2]{CLPRW25},  it suffices to verify that $\chi(\theta)=0$ for each generating relation  $\theta$.
    A direct calculation shows that
    \begin{align*}    &\chi(\genK_i\genK_j)=\chi(\genK_i)\chi(\genK_j)\langle\tau(\genK_i),\genK_j\rangle=\chi(\genK_j)\chi(\genK_i)\langle\tau(\genK_j),\genK_i\rangle=\chi(\genK_j\genK_i),\\    &\chi(\genK_i\genE_j)=\chi(\genK_i)\chi(\varrho^{p(j)})\langle\tau(\genK_i),\genE_j\rangle+\chi(\genK_i)\chi(\genE_j)\langle\tau(\genK_i),\genK_j\rangle=0,\\
    &\chi(\genE_j\genK_i)=\chi(\genE_j)\chi(\genK_i)\langle\varrho^{p(j)},\genK_i\rangle+\chi(\genK_j)\chi(\genK_i)\langle\tau(\genE_j),\genK_i\rangle=0,\\
    &{\chi(\varrho^2)=\chi(\varrho)^2\langle\tau\varrho,\varrho\rangle=1},\\
    &\chi(\genK_i\varrho)=\chi(\genK_i)\chi(\varrho)\langle \tau(\genK_i),\varrho\rangle= \chi(\varrho)\chi(\genK_i)\langle\tau\varrho,\genK_i\rangle=\chi(\varrho\cdot\genK_i),\\& \chi(\genE_i\varrho)= \chi(\genK_i)\chi(\varrho)\langle\tau(\genE_i),\varrho\rangle=0=\chi(\varrho)\chi(\varrho^{p(i)})\langle\tau\varrho,\varrho(\genE_i)\rangle = \chi(\varrho\cdot\varrho(\genE_i)).
    \end{align*}

    As for the higher Serre relations in Proposition \ref{P:SerreRelations}, we first have 
    \begin{align*}
        \chi(\genE_i\genE_j)
        &=\chi(\genE_i)\chi(\varrho^{p(j)})\langle\varrho^{p(i)},\genE_j\rangle+\chi(\genK_i)\chi(\varrho^{p(j)})\langle\tau(\genE_i),\genE_j\rangle\\
        &\qquad+\chi(\genE_i)\chi(\genE_j)\langle\varrho^{p(i)},\genK_j\rangle+\chi(\genK_i)\chi(\genE_j)\langle\tau(\genE_i),\genK_j\rangle\\
        &=\chi(\genK_i)\chi(\varrho^{p(j)})\langle\tau(\genE_i),\genE_j\rangle.
    \end{align*}
Thus $\chi(\genE_i^2)=0$ for $i\in \mathbb{I}_\iso$. For $i\in\mathbb{I}_\one$ with $i\neq \tau i$, we have 
$q_i=q_{\tau i}^{-1}$. It follows that $\chi(\genE_i\genE_j)=-\chi(\genE_j\genE_i)$ 
for $i,j\in \mathbb{I}_\one$ with $a_{i,j}=0$. The argument in \cite[Lemma 4.2]{CLPRW25} also applies to case (N-Iso). 
    
    Let $\eta$ be a monomial occurring in one of the higher quantum
Serre relations in Proposition 2.1, excluding cases (Iso) and
(N-Iso). Observe that there exists $i_v$ in $\eta=E_{i_1}\cdots E_{i_l}$ such that 
\[
   \tau(i_v)\notin\{i_u\mid 1\leq u\leq l,\ u\neq v\}.
\]
 By the definition of $\chi$ and \eqref{hopfstructure}, a straightforward calculation shows that
\begin{align*}
\chi(\eta)
&=\chi(\prod\limits_{u=1}^{v-1}\genE_{i_u}\cdot \genE_{i_v}\cdot\prod\limits_{u=v+1}^{l}\genE_{i_u})\\
&=\sum\chi((\prod\limits_{u=1}^{v-1}\genE_{i_u})_{(1)})\chi((\genE_{i_v})_{(2)})\chi((\prod\limits_{u=v+1}^{l}\genE_{i_u})_{(3)})
\langle \tau(\genE_{i_v})_{(3)},(\prod\limits_{u=v+1}^{l}\genE_{i_u})_{(2)}\rangle\\&\qquad\quad
\cdot\langle\tau(\prod\limits_{u=1}^{v-1}\genE_{i_u})_{(2)} ,(\genE_{i_v})_{(1)}\rangle
\langle\tau(\prod\limits_{u=1}^{v-1}\genE_{i_u})_{(3)}, (\prod\limits_{u=v+1}^{l}\genE_{i_u})_{(1)}\rangle\\
&=\sum\chi((\prod\limits_{u=1}^{v-1}\genE_{i_u})_{(1)})\chi(\genK_{i_v})\chi((\prod\limits_{u=v+1}^{l}\genE_{i_u})_{(3)})
\langle \tau(\genE_{i_v}),(\prod\limits_{u=v+1}^{l}\genE_{i_u})_{(2)}\rangle\\&\qquad\quad
\cdot\langle\tau(\prod\limits_{u=1}^{v-1}\genE_{i_u})_{(2)} ,\genK_{i_v}\rangle
\langle\tau(\prod\limits_{u=1}^{v-1}\genE_{i_u})_{(3)}, (\prod\limits_{u=v+1}^{l}\genE_{i_u})_{(1)}\rangle\\
&\qquad+\sum\chi((\prod\limits_{u=1}^{v-1}\genE_{i_u})_{(1)})\chi(\varrho^{p(i_v)})\chi((\prod\limits_{u=v+1}^{l}\genE_{i_u})_{(3)})
\langle \tau\varrho^{p(i_v)},(\prod\limits_{u=v+1}^{l}\genE_{i_u})_{(2)}\rangle\\&\qquad\quad
\cdot\langle\tau(\prod\limits_{u=1}^{v-1}\genE_{i_u})_{(2)} ,\genE_{i_v}\rangle
\langle\tau(\prod\limits_{u=1}^{v-1}\genE_{i_u})_{(3)}, (\prod\limits_{u=v+1}^{l}\genE_{i_u})_{(1)}\rangle\\
&=0,
\end{align*}
where the last equality follows from $\langle \tau(\genE_{i_v}),(\prod\limits_{u=v+1}^{l}\genE_{i_u})_{(2)}\rangle=0$ and $\langle\tau(\prod\limits_{u=1}^{v-1}\genE_{i_u})_{(2)} ,\genE_{i_v}\rangle=0$. 
Therefore, the proof is complete.
\end{proof}

Define a degree function $wt: {\hat{\mathbf{U}}}_q^{\geq0}\to \Z$ by setting
\begin{align*}
    wt(\genE_i)=1,\qquad wt(\genK_i)=0=wt(\varrho),\qquad i\in\mathbb{I}.
\end{align*}
On the tensor product ${\hat{\mathbf{U}}}_q^{\geq0}\otimes{\hat{\mathbf{U}}}_q^{\geq0}$, we use the degree induced by the second tensor factor, i.e. 
\begin{align*}
    wt(x\otimes y)=wt(y),\qquad \mbox{for}~x\otimes y\in{\hat{\mathbf{U}}}_q^{\geq0}\otimes{\hat{\mathbf{U}}}_q^{\geq0}.
\end{align*}


By \cite[Proposition 2.12]{CLPRW25}, we have the following.
\begin{prp}\label{xi}
    There exists an  injective algebra homomorphism $$\xi^{\imath}_{\tau}: ({\hat{\mathbf{U}}}_q^{\geq0})^{ \imath}_{\tau}\to {({\hat{\mathbf{U}}}_q^{\geq0}\otimes{\hat{\mathbf{U}}}_q^{\geq0})^{\imath}}$$ given by
    \begin{align*}
       & \genE_i\mapsto (1\otimes\varrho^{p(i)})\circ(\varrho^{p(i)}\otimes1)\circ[(\sqrt{-1})^{p(i)}(1\otimes\genE_i\varrho^{p(i)})+(\varrho^{p(i)}\tau(\genE_i)\otimes\varrho^{p(i)})\circ(1\otimes\genK_i)],\\
       &\genK_i\mapsto
         (1\otimes\varrho^{p(i)})\circ(\varrho^{p(i)}\otimes1)\circ[(-1)^{p(i)}\varrho^{p(i)}\tau(\genK_i)\otimes\varrho^{p(i)}]\circ(1\otimes\genK_i),\\
        &\varrho\mapsto-\sqrt{-1}(1\otimes\varrho)\circ(\varrho\otimes1).
    \end{align*}
\end{prp}

\begin{proof}
   Recall from \cite[Proposition 2.12]{CLPRW25} that, for any $a\in ({\hat{\mathbf{U}}}_q^{\geq0})^{ \imath}_\tau$, the map
   \begin{align*}
     \xi^{\imath}_{\tau}(a)=\sum\chi(a_{(2)})\cdot \tau(a_{(3)})\otimes a_{(1)} 
   \end{align*}   
   is an algebra homomorphism. Then applying \eqref{multiplication} and Proposition \ref{chi}, one can check that 
    \begin{align*}
   \xi^{\imath}_\tau(\genE_i) &=\chi(\varrho^{p(i)})\varrho^{p(i)}\otimes\genE_i+\chi(\genK_i)\cdot\tau(\genE_i)\otimes\genK_i\\
    &=(\sqrt{-1})^{p(i)}\varrho^{p(i)}\otimes\genE_i +q_i^{a_{i,\tau i}}\tau(\genE_i)\otimes\genK_i\\
    &=(1\otimes\varrho^{p(i)})\circ(\varrho^{p(i)}\otimes1)\circ[(\sqrt{-1})^{p(i)}(1\otimes\genE_i\varrho^{p(i)})+(\varrho^{p(i)}\tau(\genE_i)\otimes\varrho^{p(i)})\circ(1\otimes\genK_i)], \\
     \xi^{\imath}_\tau(\genK_i)
    &=\chi(\genK_i)\tau(\genK_i)\otimes\genK_i\\ &=q_i^{a_{i,\tau i}}\tau(\genK_i)\otimes\genK_i\\
    &=(1\otimes\varrho^{p(i)})\circ(\varrho^{p(i)}\otimes1)\circ[(-1)^{p(i)}\varrho^{p(i)}\tau(\genK_i)\otimes\varrho^{p(i)}]\circ(1\otimes\genK_i),\\
    \xi^{\imath}_\tau(\varrho)
    &=\chi(\varrho)\tau(\varrho)\otimes\varrho=\sqrt{-1}\varrho\otimes\varrho
    =-\sqrt{-1}\langle\varrho,\varrho\rangle\varrho\otimes\varrho
    =-\sqrt{-1}(1\otimes\varrho)\circ(\varrho\otimes1).
    \end{align*}
    
    Under the degree function, $\xi^{\imath}_\tau(\genE_i)$ has a leading term $(\sqrt{-1})^{p(i)}\varrho^{p(i)}\otimes\genE_i$. We claim that $\xi^{\imath}_\tau$ is injective. Indeed, choose a  homogeneous PBW basis $\mathcal B^+$ of
${\hat{\mathbf{U}}}_q^{+}$, and let $\mathcal B^0$ be the standard basis of
${\hat{\mathbf{U}}}_q^{0}$ consisting of monomials in the $K_i^{\pm1}$ and
$\varrho$.  By the triangular decomposition, the set
\[
   \mathcal B:=\{bu\mid b\in\mathcal B^+,\ u\in\mathcal B^0\}
\]
is a basis of ${\hat{\mathbf{U}}}_q^{\geq0}$. Then it follows inductively that, for every $b\in\mathcal B^+$
and $u\in\mathcal B^0$, $\xi^\imath_\tau(bu)$ has a leading term $c_{b,u} g_{b,u}\otimes bu$ together with terms of strictly lower degree,
where $c_{b,u}\in\mathbb C(q)^*$ and
$g_{b,u}\in{\hat{\mathbf{U}}}_q^{0}$ is invertible.

Hence the images $\xi^\imath_\tau(bu)$ ($b\in \mathcal B^+,u\in\mathcal B^0$) are linearly independent since the leading terms of the elements with maximal degrees will not disappear.
Consequently, $\xi^\imath_\tau$ is injective.
\end{proof}

In this section,  we specialize the parameter $\varsigma_i\in \C(q)^*$ in Definition \ref{def:Ui} as follows:
\begin{align*}
   \varsigma_i=\begin{cases}
    (-\sqrt{-1})^{p(i)}, &\mbox{if}~ \tau i=i,\\
    (\sqrt{-1})^{p(i)}, &\mbox{if}~ \tau i\neq i.\\
   \end{cases} 
\end{align*}

\begin{thm}\label{Phii}
    There is an algebra isomorphism $\Phi^{\imath}_\tau:({\hat{\mathbf{U}}}_q^{\geq0})^{ \imath}_\tau\to\mathbf{U}^{\imath}$ given by
    \begin{align*}
    \genE_i\mapsto (\sqrt{-1})^{p(i)}\tilde{\varrho}^{p(i)}B_i, \quad
   \genK_i\mapsto \tilde{\varrho}^{p(i)}\tilde{k}_{\tau i},\quad \varrho\mapsto-\sqrt{-1}\tilde{\varrho}.
    \end{align*}
\end{thm}
\begin{proof}
By Theorem \ref{Phi} and Proposition \ref{xi}, a direct calculation shows that
\begin{align*}
    {\Phi}^{-1}\cdot \xi^{\imath}_\tau(\genE_i)
    &=\varrho^{p(i)}\varrho'^{p(i)}[(\sqrt{-1})^{p(i)}\genF_i+\tau(\genE_i)\genK'_i]
    \\&=(\sqrt{-1})^{p(i)}\varrho^{p(i)}\varrho'^{p(i)}B_i
    =(\sqrt{-1})^{p(i)}\tilde{\varrho}^{p(i)}B_i,\\
      {\Phi}^{-1}\cdot \xi^{\imath}_\tau(\genK_i)
    &=\varrho^{p(i)}\varrho'^{p(i)}\genK_{\tau i}\genK'_i
    =\tilde{\varrho}^{p(i)}\tilde{k}_{\tau i},\\
    {\Phi}^{-1}\cdot \xi^{\imath}_\tau(\varrho)
    &=-\sqrt{-1}\varrho\varrho'=-\sqrt{-1}\tilde{\varrho}.  
    \end{align*} 
The displayed formulas show that the image of ${\Phi}^{-1}\cdot \xi^{\imath}_\tau$ contains all generators of ${\mathbf U}^\imath$. Thus, there exists a  surjective algebra homomorphism $\Phi^{\imath}_\tau:({\hat{\mathbf{U}}}_q^{\geq0})^{ \imath}_\tau\to\mathbf{U}^{\imath}$, 
such that the following diagram commutes:
\begin{equation}
\begin{tikzcd}
({\hat{\mathbf{U}}}_q^{\geq0})^{ \imath}_\tau \arrow[r, "{\xi^{ \imath}_\tau}"] \arrow[d, dashed, "{\Phi}^{\imath}_\tau"] 
& {({\hat{\mathbf{U}}}_q^{\geq0} \otimes {\hat{\mathbf{U}}}_q^{\geq0})^{ \imath}} \arrow[d, "{\Phi}^{-1}"] \\
\Ui \arrow[r, hookrightarrow] 
& {{\mathbf{U}}}.
\end{tikzcd} 
\end{equation}
Since ${\Phi}$ is an isomorphism and ${\xi^{ \imath}_\tau}$ is an algebra embedding, $\Phi^{\imath}_\tau$ is also injective and hence an isomorphism.
\end{proof}
\begin{rmk}
    If  $i\in \mathbb{I}_\zero$ and $i\neq j$, then applying the formula from Theorem \ref{formula} to  (N-Iso) in Proposition \ref{P:SerreRelations} and using Theorem \ref{Phii}, we see that the relations coincide with those in \cite[Theorem 3.1(3.6)-(3.9)]{CLW21} and \cite[Theorem 4.2(4.9)-(4.10)]{Ch19}.
\end{rmk}

\section{Serre relations for \texorpdfstring{$\mathrm i$}{i}quantum supergroups}\label{application}

This section is devoted to computing the Serre relations for quasi-split {$\mathrm i$}quantum supergroups using the formula established in Theorem \ref{formula}, and to introducing a new notation to simplify their presentation.
In particular, all the Serre relations that occur in quasi-split iquantum supergroups of basic type are included.

\subsection{Serre relations for the {$\mathrm i$}Hopf algebra $({\hat{\mathbf{U}}}_q^{\geq0})_\tau^{ \imath}$}

Define the $q$-commutator on the iHopf algebra $({\hat{\mathbf{U}}}_q^{\geq0})_\tau^{ \imath}$ by 
\begin{equation}
    \begin{aligned}
        \Ad_q u(v)=u\circ v-(-1)^{p(u)p(v)}q^{(\al,\beta)}v\circ u,
    \end{aligned}
\end{equation}
with $u\in ({\hat{\mathbf{U}}}_q^{\geq0})^{ \imath}_{\tau,\al}$, $v\in({\hat{\mathbf{U}}}_q^{\geq0})^{ \imath}_{\tau,\beta}$.

\begin{prp}\label{trans1}
The following relations hold in $({\hat{\mathbf{U}}}_q^{\geq0})^{ \imath}_\tau$:
\begin{enumerate}
    \item For $i,j\in \mathbb{I}_\one$ with $a_{i,j}=0$,
    \begin{align*}
       \genE_i\circ\genE_{j}+\genE_{j}\circ\genE_i= \delta_{\tau i,j}
        \frac{1}{q_i-q_i^{-1}}\varrho\circ(\genK_j-\genK_i).
    \end{align*}
      
    \item {For $i\in \mathbb{I}_\zero\cup \mathbb{I}_\niso$ and $\tau i=i$,}
\begin{align*}
&\Ad_q \genE_i^{1-a_{i,j}}(\genE_j)
\\&=
\begin{cases}
0, &a_{i,j}=0,\\
(-1)^{p(i)}q_i\varrho^{p(i)}\circ\genK_i\circ\genE_j, & a_{i,j}=-1, \\
(q_i^2+(-1)^{p(i)})[2]_i(\genE_i\circ\genE_j-(-1)^{p(i)+p(i)p(j)}\genE_j\circ\genE_i)\circ\varrho^{p(i)}\circ\genK_i, & a_{i,j}=-2,\\
s_1(i,j)\genE_i\circ\genE_i\circ\genE_j\circ\varrho^{p(i)}\circ\genK_i    +s_2(i,j)\genE_j\circ\genE_i\circ\genE_i\circ\varrho^{p(i)}\circ\genK_i\\\quad
    +(-1)^{p(i)p(j)}s_3(i,j)\genE_i\circ\genE_j\circ\genE_i\circ\varrho^{p(i)}\circ\genK_i+s_4(i,j)\genK_i\circ\genK_i \circ\genE_j, &a_{i,j}=-3,
\end{cases}
\end{align*} 
where
{\small\begin{align*}
  &s_1(i,j)=-\Big(q_i+(-1)^{p(i)}q_i^{-1}\Big)[3]_i-(-1)^{p(i)}(2q_i-q_i^5)-q_i^3,\\
  &s_2(i,j)=-\Big(q_i^3+(-1)^{p(i)}q_i+ q_i^{-1}\Big)[3]_i-q_i+(1-(-1)^{p(i)})q_i^{-1},\\
  &s_3(i,j)=\Big(3+3(-1)^{p(i)}q_i^2+(-1)^{p(i)}q_i^{-2}+q_i^4\Big)[3]_i-2-(-1)^{p(i)}q_i^2,\\
  &s_4(i,j)=(-1)^{p(i)}q_i^2 [3]_i\Big( [3]_i-(1-(-1)^{p(i)})\Big).
\end{align*}}
\end{enumerate}
  
\end{prp}
\begin{proof}
    We only prove  part (2) for $a_{i,j}=-1$; the remaining cases follow by similar arguments and are hence omitted. For $a_{i,j}=-1$ and $\tau i=i$, we have 
\begin{align*}
    \ad_q\genE_i^2(\genE_j)
    &=\ad_q\genE_i(\genE_i\genE_j-(-1)^{p(i)p(j)}q_i^{-1}\genE_j\genE_i)
    \\
    &=\genE_i^2\genE_j-(-1)^{p(i)p(j)}q_i^{-1}\genE_i\genE_j\genE_i-(-1)^{p(i)+p(i)p(j)}q_i\genE_i\genE_j\genE_i+(-1)^{p(i)}\genE_j\genE_i^2.
\end{align*}
Applying Theorem \ref{formula}, we obtain 
\begin{align*}
    &\genE_i^2\genE_j =\genE_i\circ\genE_i\circ\genE_j+\frac{1}{q_i-q_i^{-1}}\varrho^{p(i)}\circ\genK_i\circ\genE_j,\\
    &\genE_i\genE_j\genE_i =\genE_i\circ\genE_j\circ\genE_i+(-1)^{p(i)p(j)}\frac{q_i}{q_i-q_i^{-1}}\varrho^{p(i)}\circ\genK_i\circ\genE_j,\\
    &\genE_j\genE_i^2 =\genE_j\circ\genE_i\circ\genE_i+\frac{1}{q_i-q_i^{-1}}\varrho^{p(i)}\circ\genK_i\circ\genE_j.
\end{align*}
Thus we have
\begin{align*}
    \Ad_q\genE_i^2(\genE_j)
    &=\ad_q\genE_i^2(\genE_j)+\frac{(-1)^{p(i)}q_i^2-(-1)^{p(i)}}{q_i-q_i^{-1}}\varrho^{p(i)}\circ\genK_i\circ\genE_j
    \\
    &=(-1)^{p(i)}q_i\varrho^{p(i)}\circ\genK_i\circ\genE_j.
\end{align*}
\end{proof}

\begin{prp}\label{trans2}
The following holds in $({\hat{\mathbf{U}}}_q^{\geq0})^{ \imath}_\tau$ whenever the given  subdiagram appears. 
\begin{enumerate}
    \item For  
$$\hspace{.75in}\xy
(-10,0)*{\otimes};(0,0)*{\otimes}**\dir{-};(0,0)*{\otimes};(10,0)*{\textup\fullmoon}**\dir{-};
(-10,-4)*{\scriptstyle i};(0,-4)*{\scriptstyle j};(10,-4)*{\scriptstyle k};
\endxy\quad( \tau i=j, \tau k\neq k)$$ 
$$
\Ad_q \genE_j\circ\Ad_{q}\genE_k\circ \Ad_{q} \genE_j (\genE_i)=q_j^{-2}\varrho\circ\genK_j\circ\Ad_q\genE_j(\genE_k)-\varrho\circ\genK_i\circ\Ad_q \genE_k(\genE_j).
$$
\item For  
$$\hspace{.75in}\xy
(-10,0)*{\textup\fullmoon};(0,0)*{\otimes}**\dir{-};(0,0)*{\otimes};(10,0)*{\otimes}**\dir{-};
(-10,-4)*{\scriptstyle i};(0,-4)*{\scriptstyle j};(10,-4)*{\scriptstyle k};
\endxy\quad( \tau j=k, \tau i\neq i)$$ 
$$
\Ad_q \genE_j\circ\Ad_{q}\genE_k\circ \Ad_{q} \genE_j (\genE_i)=q_j^{2}\varrho\circ\genK_j\circ\Ad_q\genE_j(\genE_i)-\varrho\circ\genK_i\circ\Ad_q \genE_i(\genE_j).
$$
\end{enumerate}

\end{prp}
\begin{proof}
For the case $p(i)=p(j)=1,p(k)=0$ in part (1), the corresponding  Cartan matrix is
\begin{align*}
    \begin{pmatrix}
        0 &1 &0 \\
        -1 &0 &1\\
        0 &-1 &2\\
    \end{pmatrix}.
\end{align*}
  By Proposition \ref{P:SerreRelations}, a direct calculation shows that
\begin{align*}
\ad_q \cdot&\genE_j \cdot\ad_{q}\genE_k \cdot \ad_{q} \genE_j (\genE_i)\\
    &=\genE_{j} \genE_k \genE_{j} \genE_i
    -(-1)^{p(i)}(q_j+q_j^{-1})\genE_{j} \genE_k\genE_i \genE_{j}
    +(-1)^{p(k)(p(i)+1)+p(i)}\genE_{j}\genE_i \genE_{j} \genE_k\\&\qquad
    -(-1)^{p(k)+p(i)+1}\genE_k\genE_{j} \genE_i  \genE_{j}
    +(-1)^{p(k)p(i)}\genE_i\genE_{j} \genE_k \genE_{j}
    -(-1)^{p(k)(p(i)+1)}q_j^{-1}\genE_{j}^2 \genE_i \genE_k\\&\qquad
    -(-1)^{p(k)}q_j\genE_k \genE_i \genE_{j}^2\\
    &=\genE_{j} \genE_k \genE_{j} \genE_i
    -(-1)^{p(i)}(q_j+q_j^{-1})\genE_{j} \genE_k\genE_i \genE_{j}
    +(-1)^{p(k)(p(i)+1)+p(i)}\genE_{j}\genE_i \genE_{j} \genE_k\\&\qquad
    -(-1)^{p(k)+p(i)+1}\genE_k\genE_{j} \genE_i  \genE_{j}
    +(-1)^{p(k)p(i)}\genE_i\genE_{j} \genE_k \genE_{j}.  
\end{align*}

For $\tau i=j$, note that 
\begin{align*}
&\langle \tau(\genE_i),\genE_j\rangle=(-1)^{p(i)}\langle \genE_{\tau i},\genE_j\rangle=\frac{1}{q_j-q_j^{-1}},\\
&\langle \tau(\genE_j),\genE_i\rangle=(-1)^{p(j)}\langle \genE_{\tau j},\genE_i\rangle=\frac{1}{q_i-q_i^{-1}}=-\frac{1}{q_j-q_j^{-1}}.
\end{align*}
{Then using Theorem \ref{formula} we have }
\begin{align*}
    &\genE_{j} \genE_k \genE_{j} \genE_i=\genE_{j} \circ\genE_k\circ \genE_{j} \circ\genE_i-\frac{1}{q_j-q_j^{-1}}(-q_j\varrho\circ\genK_i\circ\genE_k\circ \genE_j+q_j^{2}\varrho\circ\genK_i\circ\genE_j\circ \genE_k),\\
    &\genE_{j} \genE_k\genE_i \genE_{j}=\genE_{j}\circ \genE_k\circ\genE_i \circ\genE_{j}-\frac{1}{q_j-q_j^{-1}}(\varrho\circ\genK_i\circ\genE_k\circ \genE_j-q_j^{-2}\varrho\circ\genK_j\circ\genE_j\circ \genE_k),\\
    &\genE_{j}\genE_i \genE_{j} \genE_k=\genE_{j}\circ\genE_i\circ \genE_{j} \circ\genE_k-\frac{1}{q_j-q_j^{-1}}(\varrho\circ\genK_i\circ\genE_j\circ \genE_k-q_j^{-1}\varrho\circ\genK_j\circ\genE_j\circ \genE_k),\\
    &\genE_k\genE_{j} \genE_i  \genE_{j}=\genE_k\circ\genE_{j} \circ\genE_i \circ \genE_{j}-\frac{1}{q_j-q_j^{-1}}(q_j\varrho\circ\genK_i\circ\genE_k\circ \genE_j-q_j^{-2}\varrho\circ\genK_j\circ\genE_k\circ \genE_j),\\
    &\genE_i\genE_{j} \genE_k \genE_{j}=\genE_i\circ\genE_{j}\circ \genE_k\circ \genE_{j}+\frac{1}{q_j-q_j^{-1}}(\varrho\circ\genK_j\circ\genE_k\circ \genE_j-q_j^{-1}\varrho\circ\genK_j\circ\genE_j\circ \genE_k).
\end{align*}
{Thus we obtain} 
   \begin{align*}
    \Ad_q \genE_j&\circ\Ad_{q}\genE_k\circ \Ad_{q} \genE_j (\genE_i)\\
    &=\ad_q \cdot\genE_j \cdot\ad_{q}\genE_k \cdot \ad_{q} \genE_j (\genE_i)
   +\frac{1}{q_j-q_j^{-1}}(-q_j\varrho\circ\genK_i\circ\genE_k\circ \genE_j+q_j^{2}\varrho\circ\genK_i\circ\genE_j\circ \genE_k)\\
    &\quad     +\frac{q_j+q_j^{-1}}{q_j-q_j^{-1}}(\varrho\circ\genK_i\circ\genE_k\circ \genE_j-q_j^{-2}\varrho\circ\genK_j\circ\genE_j\circ \genE_k)\\
    &\quad     
   -\frac{1}{q_j-q_j^{-1}}(\varrho\circ\genK_i\circ\genE_j\circ \genE_k-q_j^{-1}\varrho\circ\genK_j\circ\genE_j\circ \genE_k)\\
    &\quad     
   -\frac{1}{q_j-q_j^{-1}}(q_j\varrho\circ\genK_i\circ\genE_k\circ \genE_j-q_j^{-2}\varrho\circ\genK_j\circ\genE_k\circ \genE_j)\\
    &\quad     
   -\frac{1}{q_j-q_j^{-1}}(\varrho\circ\genK_j\circ\genE_k\circ \genE_j-q_j^{-1}\varrho\circ\genK_j\circ\genE_j\circ \genE_k)\\
    &=\varrho\circ\genK_i\circ(q_j\genE_j\circ \genE_k-\genE_k\circ \genE_j)
    +q_j^{-1}\varrho\circ\genK_j\circ(q_j^{-1}\genE_j\circ \genE_k-\genE_k\circ \genE_j)\\
    &=q_j^{-2}\varrho\circ\genK_j\circ\Ad_q\genE_j(\genE_k)-\varrho\circ\genK_i\circ\Ad_q \genE_k(\genE_j).
\end{align*} 
Similarly, the remaining cases can be verified.
\end{proof}

\begin{prp}\label{trans3}
For the subdiagram
$$\xy
(0,0)*{\textup\fullmoon};(7,5)*{\otimes}**\dir{-};(0,0)*{\textup\fullmoon};(7,-5)*{\otimes}**\dir{-};(7,5)*{\otimes};(7,-5)*{\otimes}**\dir{=};
(-4,0)*{\scriptstyle i};(10,5)*{\scriptstyle j};(10,-5)*{\scriptstyle k};
\endxy\quad(\tau j=k,\tau i=i),$$
the superalgebra $({\hat{\mathbf{U}}}_q^{\geq0})^{ \imath}_\tau$ satisfies
$$\Ad_q \genE_k\circ \Ad_q \genE_j (\genE_i)-\Ad_q \genE_j \circ \Ad_q \genE_k(\genE_i)=-q_j^{-1}\varrho\circ(\genK_k-\genK_j)\circ\genE_i.$$
\end{prp}
\begin{proof}
    For the given subdiagram, the corresponding Cartan matrix is 
    \begin{align*}
    \begin{pmatrix}
        2 &-1 &-1 \\
        1 &0 &-2\\
        -1 &2 &0\\
    \end{pmatrix}.
\end{align*}
{For $\tau j=k, \tau i=i$, by Proposition \ref{P:SerreRelations} and Theorem \ref{formula}, a direct calculation shows that}
\begin{align*}
 \Ad_q \genE_k&\circ \Ad_q \genE_j (\genE_i)-\Ad_q \genE_j \circ \Ad_q \genE_k(\genE_i)\\
 &=\ad_q \genE_k \ad_q \genE_j (\genE_i)-\ad_q \genE_j \ad_q \genE_k(\genE_i)\\
 &\qquad
 +\frac{q_j^{-2}-1}{q_j-q_j^{-1}}\varrho\circ \genK_k\circ\genE_i+\frac{q_j^{-2}-1}{q_j-q_j^{-1}}\varrho\circ\genK_j\circ\genE_i\\
 &=-q_j^{-1}\varrho\circ(\genK_k-\genK_j)\circ\genE_i.
\end{align*}    
\end{proof}

We conclude this subsection by deriving several relations for the iHopf algebra $({\hat{\mathbf{U}}}_q^{\geq0})^{ \imath}_\tau$ of Kac–Moody type.
For the subdiagram 
$$
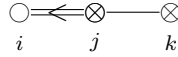
\xy
(-10,0)*{\fullmoon};(0,0)*{\otimes}**\dir{=};(0,0)*{\otimes};(10,0)*{\otimes}**\dir{-};(-5,0)*{<};
(-10,-4)*{\scriptstyle i};(0,-4)*{\scriptstyle j};(10,-4)*{\scriptstyle k};
\endxy$$
and 
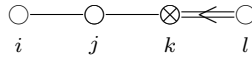
$$\xy
(-10,0)*{\fullmoon};(0,0)*{\fullmoon}**\dir{-};(0,0)*{\fullmoon};(10,0)*{\otimes}**\dir{-};(15,0)*{<};(10,0)*{\otimes};(20,0)*{\fullmoon}**\dir{=};
(-10,-4)*{\scriptstyle i};(0,-4)*{\scriptstyle j};(10,-4)*{\scriptstyle k};(20,-4)*{\scriptstyle l};
\endxy$$
 the corresponding Cartan matrices are 
\begin{align*}
      \begin{pmatrix}
        2 &-2 &0 \\
        -1 &0 &1\\
        0 &-1 &0\\
    \end{pmatrix},  
    \qquad 
    \begin{pmatrix}
        2 &-1 &0 &0 \\
        -1 &2 &-1&0\\
        0 &-1 &0&2\\
        0 &0 &-1&2\\
    \end{pmatrix}.
    \end{align*}
Applying Theorem \ref{formula} to 
the defining relations of affine quantum supergroups given in \cite{Ya99}, 
we obtain the corollary as follows.

\begin{cor}{The following relations hold in $({\hat{\mathbf{U}}}_q^{\geq0})^{ \imath}_\tau$:}
    \begin{enumerate}
        \item For 
        $$\xy
(-10,0)*{\textup\fullmoon};(0,0)*{\otimes}**\dir{=};(0,0)*{\otimes};(10,0)*{\otimes}**\dir{-};(-5,0)*{<};
(-10,-4)*{\scriptstyle i};(0,-4)*{\scriptstyle j};(10,-4)*{\scriptstyle k};
\endxy\quad (\tau j =k)$$
$$
\Ad_q \genE_j\circ\Ad_{q}\genE_k\circ \Ad_{q} \genE_j (\genE_i)
=q_j^{2}\varrho\circ\genK_j\circ\Ad_q\genE_j(\genE_i)-\varrho\circ\genK_k\circ\Ad_q \genE_i(\genE_j).
$$

\item {For } 
$$\xy
(-10,0)*{\textup\fullmoon};(0,0)*{\textup\fullmoon}**\dir{-};(0,0)*{\textup\fullmoon};(10,0)*{\otimes}**\dir{-};(15,0)*{<};(10,0)*{\otimes};(20,0)*{\textup\fullmoon}**\dir{=};
(-10,-4)*{\scriptstyle i};(0,-4)*{\scriptstyle j};(10,-4)*{\scriptstyle k};(20,-4)*{\scriptstyle l};
\endxy\quad (\tau i=j)$$
\begin{align*}
 \Ad_q \genE_k&\circ \Ad_q \genE_j \circ \Ad_q \genE_k\circ \Ad_q \genE_l\circ \Ad_q \genE_k\circ \Ad_q \genE_j (\genE_i)\\&=-(q_j-q_j^{-1})\genK_j\circ(q_j^{-1}\genE_k\circ\genE_j\circ\genE_k\circ\genE_l\circ\genE_k+\genE_k\circ\genE_l\circ\genE_k\circ\genE_j\circ\genE_k).   
\end{align*}

    \end{enumerate}
\end{cor}


\subsection{Reformulation of Serre relations}
 To encode the matching structure explicitly, we now introduce formal symbols $\mathfrak{e}_{i_{t}}~(1\leq t\leq l)$ as placeholders for pairing positions, marking which factors are to be paired.

 For $T_{k}=\{t_1<\cdots< t_k\}\subseteq \{1,\cdots,l\}$, define the operator
 $$\mathfrak{e}_{T_k}(\genE_{i_1}\circ\cdots\circ\genE_{i_l})=(\mathop{\circ\prod}\limits_{u=1}^{t_1 -1}\genE_{i_u}) \circ\mathfrak{e}_{i_{t_1}}\circ (\mathop{\circ\prod}\limits_{u=t_1 +1}^{t_2 -1}\genE_{i_u}) \circ\mathfrak{e}_{i_{t_2}}\circ\cdots\circ\mathfrak{e}_{i_{t_k}}
 \circ(\mathop{\circ\prod}\limits_{u=t_k +1}^{l}\genE_{i_u}),$$
 where $\mathfrak{e}_{i_{t}}~(t\in T_k)$ are formal pairing symbols (not algebra generators) marking a pairing  at position $t$.
 For an even integer $k$, let $\mathcal{M}_{\frac{k}{2}}(T_k):=\{M\in\mathcal M_{\frac{k}{2}}(S)\mid T(M)=T_k\}$. For each $M\in \mathcal{M}_{\frac{k}{2}}(T_k)$, apply the matching rule
 \begin{align}\label{evaluate}
    (\mathfrak{e}_{i_{x}}, \mathfrak{e}_{i_y}) \mapsto \varphi(x,y)\cdot \varrho^{p(i_{x})}\circ \genK_{i_{y}}, \qquad (t_x, t_y)\in M.
 \end{align}
 Then $\mathfrak{e}_{T_k}(\genE_{i_1}\circ\cdots\circ\genE_{i_l})$ evaluates to 
 \begin{align*}
     \sum\limits_{M\in \mathcal{M}_{\frac{k}{2}}(T_k)}(-1)^{\frac{k}{2}}\Big(\mathop{\circ\prod}\limits_{{(x,y)\in M}}\varphi(x,y)\cdot 
    \varrho^{p(i_{x})}\circ \genK_{i_{y}}\Big) \circ E_{I^M}^\circ. 
 \end{align*}
 For odd $k$, the unpaired symbol cannot be evaluated, so $\mathfrak{e}_{T_k}(\genE_{i_1}\circ\cdots\circ\genE_{i_l})=0$.
\begin{rmk} 
Summing over all $k$-subsets yields
\begin{align*}
   \mathcal{F}_{k} := \sum\limits_{1\leq t_1<t_2\cdots<t_k\leq l}\mathfrak{e}_{T_k}(\genE_{i_1}\circ\cdots\circ\genE_{i_l}),
\end{align*}
the sum of all terms that contain precisely $k $  of the $\mathfrak{e}_{i_u}$'s  (and $l-k$  of the  $\genE_{i_u}$'s ).
Consequently,
 \begin{align*}  \genE_{i_1}\circ\cdots\circ\genE_{i_l}+\mathcal{F}_1 +\mathcal{F}_2 +\cdots +\mathcal{F}_l=(\genE_{i_1} + \mathfrak{e}_{i_1})\circ\cdots \circ(\genE_{i_l} + \mathfrak{e}_{i_l}). 
\end{align*}
Upon evaluation via ordered  matching, each $\mathcal{F}_{k}$ with $k=2s$ becomes 
$$\sum\limits_{M\in \mathcal{M}_s(S)}(-1)^s\Big(\mathop{\circ\prod}\limits_{{(x,y)\in M}}\varphi(x,y)\cdot  
    \varrho^{p(i_{x})}\circ \genK_{i_{y}}\Big) \circ E_{I^M}^\circ,$$
and then $(\genE_{i_1} + \mathfrak{e}_{i_1})\circ\cdots \circ(\genE_{i_l} + \mathfrak{e}_{i_l})$ evaluates to  
$E^{\circ}_{I}+o(E^{\circ}_{I})$.
\end{rmk}

For notational simplicity, set 
\begin{align*}
    \genE_i^{\star}=\genE_i+\mathfrak{e}_{i} \qquad(i\in\mathbb{I}),
\end{align*}
with multiplication normalized via the matching rule \eqref{evaluate}. 
The resulting relations for $\genE_i^\star$ in iquantum supergroups encode the Serre relations for $\genE_i$ in Borel quantum supergroups.

Define the $q$-commutator on $\genE_i^{\star},\genE_j^{\star}$ by
$$\Ad_q \genE_i^{\star}(\genE_j^{\star})=\genE_i^{\star}\circ \genE_j^{\star}-(-1)^{p(i)p(j)}q_i^{a_{i,j}}\genE_j^{\star}\circ \genE_i^{\star}.$$

It follows directly from the classification that a quasi-split iquantum supergroup of basic type with
$\mathbb{I}_\one\neq\emptyset$ can occur only in types $A(m,n)$, $B(m,n+1)$ and $D(m,n+1)$. Thus, we reformulate all the relations for these iquantum supergroups as follows.

\begin{thm}\label{star} 
For any $\tau$ with $\tau i\neq i$ if $i\in \mathbb{I}_\iso$, the superalgebra $({\hat{\mathbf{U}}}_q^{\geq0})_\tau^{ \imath}$ of basic type satisfies the following relations whenever the given Dynkin subdiagram appears: 

\begin{enumerate}
\item For $i,j\in \mathbb{I}_\one$ with $a_{ij}=0$, 
$$\genE_i^{\star}\circ\genE_j^{\star}=-\genE_j^{\star}\circ\genE_i^{\star}.$$
\item For $i\in \mathbb{I}_\zero\cup \mathbb{I}_\niso$ and $i\neq j$ with $a_{i,j}\in\{0,-1,-2\}$, 
$$\Ad_q \genE_i^{{\star} ^{1-a_{i,j}}}(\genE_j^{\star})=0.
$$
\item For  $sgn_{ij}\neq sgn_{jk}$,
$$\hspace{.75in}\xy
(-10,0)*{\odot};(0,0)*{\otimes}**\dir{-};(0,0)*{\otimes};(10,0)*{\odot}**\dir{-};
(-10,-4)*{\scriptstyle i};(0,-4)*{\scriptstyle j};(10,-4)*{\scriptstyle k};
\endxy$$ 
$$
\Ad_q \genE_j^{\star}\circ\Ad_{q}\genE_k^{\star}\circ \Ad_{q} \genE_j^{\star} (\genE_i^{\star})=0.
$$
\item For $\tau j=k$,
$$\xy
(0,0)*{\textup\fullmoon};(7,5)*{\otimes}**\dir{-};(0,0)*{\textup\fullmoon};(7,-5)*{\otimes}**\dir{-};(7,5)*{\otimes};(7,-5)*{\otimes}**\dir{=};
(-4,0)*{\scriptstyle i};(10,5)*{\scriptstyle j};(10,-5)*{\scriptstyle k};
\endxy\quad $$
$$\Ad_q \genE_k^{\star}\circ \Ad_q \genE_j^{\star} (\genE_i^{\star})=\Ad_q \genE_j^{\star} \circ \Ad_q \genE_k^{\star}(\genE_i^{\star}).$$
\end{enumerate}

    \end{thm}

We also obtain more Serre relations for the algebra $({\hat{\mathbf{U}}}_q^{\geq0})_\tau^{ \imath}$ beyond basic types.
\begin{prp}
If the given subdiagram occurs, then the following relations hold in the algebra $({\hat{\mathbf{U}}}_q^{\geq0})_\tau^{ \imath}$.

\begin{enumerate}
\item For $i\in \mathbb{I}_\zero\cup \mathbb{I}_\niso$ and $i\neq j$,
$$\Ad_q \genE_i^{{\star} ^{1-a_{i,j}}}(\genE_j^{\star})=0.
$$
\item For 
$$\xy
(-10,0)*{\textup\yy};(0,0)*{\otimes}**\dir{=};(0,0)*{\otimes};(10,0)*{\odot}**\dir{-};(-5,0)*{<};
(-10,-4)*{\scriptstyle i};(0,-4)*{\scriptstyle j};(10,-4)*{\scriptstyle k};
\endxy$$
$$
\Ad_q \genE_j^{\star}\circ\Ad_{q}\genE_k^{\star}\circ \Ad_{q} \genE_j^{\star} (\genE_i^{\star})=0.
$$

\item For
$$\xy
(-10,0)*{\textup\fullmoon};(0,0)*{\otimes}**\dir{=};(0,0)*{\otimes};(10,0)*{\otimes}**\dir{-};(-5,0)*{>};
(-10,-4)*{\scriptstyle i};(0,-4)*{\scriptstyle j};(10,-4)*{\scriptstyle k};
\endxy$$
$$\Ad_q \genE_j^{\star}\circ \Ad_q (\Ad_q \genE_j^{\star}(\genE_k^{\star}))\circ \Ad_q \genE_i^{\star}\circ \Ad_q \genE_j^{\star} (\genE_k^{\star})=0.$$
\item For
$$\xy
(-10,0)*{\odot};(0,0)*{\textup\fullmoon}**\dir{-};(0,0)*{\textup\fullmoon};(10,0)*{\otimes}**\dir{-};(15,0)*{<};(10,0)*{\otimes};(20,0)*{\textup\fullmoon}**\dir{=};
(-10,-4)*{\scriptstyle i};(0,-4)*{\scriptstyle j};(10,-4)*{\scriptstyle k};(20,-4)*{\scriptstyle l};
\endxy$$
$$\Ad_q \genE_k^{\star}\circ \Ad_q \genE_j^{\star} \circ \Ad_q \genE_k^{\star}\circ \Ad_q \genE_l^{\star}\circ \Ad_q \genE_k^{\star}\circ \Ad_q \genE_j^{\star} (\genE_i^{\star})=0.$$
\end{enumerate}
\end{prp}

\begin{rmk}
While this manuscript was in its final stages of preparation,  the preprint \cite{CS26} appeared, addressing closely related questions. In \cite{CS26}, Serre relations for iquantum supergroups of type $A$ were obtained independently. 
Their approach uses projection techniques to compute the remainder terms for
all type-\(A\) super Satake diagrams, including those with black nodes and
thus the non-quasi-split cases, whereas our approach derives the relations
from an iHopf-algebra realization and applies uniformly across all
basic Lie superalgebra types.
     
\end{rmk}

\vspace{2mm}
\noindent{\bf Acknowledgments} 
The authors would like to thank Ming Lu, Yaolong Shen, Weiqiang Wang and Weinan Zhang for helpful discussions.
J. Chen is partially supported by
the National Natural Science Foundation of China (No. 12601077) and 
the National Natural Science Foundation of China Tianyuan Fund for Mathematics (No. 12526562). 
S. Ruan is partially supported by
Fundamental Research Funds for Central Universities of China (No. 20720250059), Fujian Provincial Natural Science Foundation of China (No. 2024J010006) and
the National Natural Science Foundation of China (No. 12671051).  
H. Zhu is partially supported by 
the National Natural Science Foundation of China (No. 12371040).


\end{document}